\documentclass[12pt]{amsart}
\usepackage{amsmath, amscd, amssymb, amsthm,amsfonts}
\usepackage{diagrams}
\usepackage[T1]{fontenc}
\makeatletter

\usepackage{dbnsymb}
\usepackage{stmaryrd}

\newcommand{\nl}{\hfil\break}

\ifx\pdfoutput\undefined
\usepackage{graphicx}
\else
\usepackage[pdftex]{graphicx}
\usepackage{epstopdf}
\fi

\newcommand{\co}{\colon\thinspace}

\DeclareMathOperator{\tr}{Tr}

\DeclareMathOperator{\Iso}{Iso}

\DeclareMathOperator{\Cob}{Cob}
\DeclareMathOperator{\PCob}{Pre-Cob}

\DeclareMathOperator{\Kom}{Kom}
\DeclareMathOperator{\Mat}{Mat}

\DeclareMathOperator{\Morph}{Hom}
\DeclareMathOperator{\Cone}{Cone}
\DeclareMathOperator{\Inv}{Inv}

\DeclareMathOperator{\Obj}{Ob}

\DeclareMathOperator{\Univ}{U}
\DeclareMathOperator{\TL}{TL}
\DeclareMathOperator{\KF}{\FK}

\DeclareMathOperator{\CFK}{CFK}

\DeclareMathOperator{\FK}{FK}
\DeclareMathOperator{\SU}{SU}
\DeclareMathOperator{\Tot}{Tot}
\DeclareMathOperator{\Ktheory}{K}
\DeclareMathOperator{\Homology}{H}

\newcommand{\inp}[1]{\ensuremath{\langle #1 \rangle}}

\newcommand{\Mb}[1]{\ensuremath{\mathbb{#1}}}

\newcommand{\normaltext}[1]{\textnormal{#1}}
\newcommand{\quantsu}[0]{\Univ_q\mathfrak{sl}(2)}

\newcommand{\CPic}[1]{
\begin{minipage}{.45in}
\includegraphics[scale=.75]{#1}
\end{minipage}
}

\newcommand{\MPic}[1]{
\begin{minipage}{.35in}
\includegraphics[scale=.45]{#1}
\end{minipage}
}

\newcommand{\BPic}[1]{
\begin{minipage}{1in}
\includegraphics[scale=1.5]{#1}
\end{minipage}
}

\newcommand{\SPic}[1]{
\begin{minipage}{.75in}
\includegraphics[scale=1.15]{#1}
\end{minipage}
}

\theoremstyle{plain}
\newtheorem{theorem}[subsection]{Theorem}

\newtheorem{proposition}[subsection]{Proposition}
\newtheorem{corollary}[subsection]{Corollary}
\newtheorem{lemma}[subsection]{Lemma}
\theoremstyle{remark}
\newtheorem{example}[subsection]{Example}

\newtheorem*{remark}{Remark}
\theoremstyle{definition}

\newtheorem{definition}[subsection]{Definition}

\begin{document}
\title{Categorification of the Jones-Wenzl Projectors}

\author[Benjamin Cooper and Vyacheslav Krushkal]{Benjamin Cooper and Vyacheslav Krushkal$^{^*}$}

\thanks{$^{^*}$Partially supported by the NSF and by the I.H.E.S}

\address{Department of Mathematics, University of Virginia, Charlottesville, VA 22904}
\email{bjc4n\char 64 virginia.edu, krushkal\char 64 virginia.edu}

\begin{abstract}
  The Jones-Wenzl projectors $p_n$ play a central role in quantum topology,
  underlying the construction of $\SU(2)$ topological quantum field theories
  and quantum spin networks.  We construct chain complexes ${P}_n$, whose
  graded Euler characteristic is the ``classical'' projector $p_n$ in the
  Temperley-Lieb algebra. We show that the ${P}_n$ are idempotents and
  uniquely defined up to homotopy.  Our results fit within the general
  framework of Khovanov's categorification of the Jones polynomial.
  Consequences of our construction include families of knot invariants
  corresponding to higher representations of $\quantsu$ and a
  categorification of quantum spin networks. We introduce $6j$-symbols in
  this context.
\end{abstract}

\maketitle

\section{Introduction}

In \cite{MR1740682} Mikhail Khovanov introduced a categorification of the
Jones polynomial, giving rise to a new conceptual framework for quantum
invariants of links in the $3$-sphere.  The results in \cite{MR1740682} fit
in the context of categorification of the Temperley-Lieb algebra
\cite{MR1928174}, \cite{MR2174270} (also \cite{MR2726291},
\cite{MR2120117}).  Roughly speaking, categorification associates to an
algebra $A$ a category ${\mathcal C}$ whose Grothendieck group
$\Ktheory_0({\mathcal C})$ is isomorphic to $A$. Moreover, multiplication by
generators of $A$ gives rise to functors acting on ${\mathcal C}$ and
satisfying natural properties \cite{MR2559652}.  An extension from planar
Temperley-Lieb diagrams to tangles is achieved by passing from additive to
triangulated categories. The resulting link homology theory satisfies
functoriality under surface cobordisms in $4$-space, an important feature
that was not apparent at the level of its graded Euler characteristic, the
Jones polynomial.

An important open problem in the subject is to extend known
categorifications from links in the $3$-sphere to quantum invariants of
$3$-manifolds.  The constructions of the $\SU(2)$ quantum invariants by
Reshetikhin-Turaev \cite{MR1292673} and Turaev-Viro \cite{MR1191386} rely on
the Jones-Wenzl projectors $p_n$ \cite{MR1622916, MR873400}, certain special
elements of the Temperley-Lieb algebra. In the Reshetikhin-Turaev theory,
one uses the Jones-Wenzl projectors to label the components of the link in a
surgery presentation of the $3$-manifold. In the Turaev-Viro approach, a
triangulation of the $3$-manifold is assigned a state sum involving the
$6j$-symbols, an important ingredient in the theory of quantum spin
networks.  (An additional key feature of the $3$-manifold invariants,
closely related to the properties of the Jones-Wenzl projectors, is that the
``quantum'' parameter $q$ has to be specialized to a root of unity in order
to get a semisimple theory).

The main goal of this paper is to introduce a categorification of the
Jones-Wenzl projectors.  The Temperley-Lieb algebra $\TL_n$ is an algebra
over ${\mathbb Z}[q,q^{-1}]$, additively generated by planar diagrams
connecting $n$ points at the top and at the bottom of a rectangle and the
multiplication is defined on generators by vertical stacking of such
diagrams (see section \ref{TL section} below for more details).
The Jones-Wenzl projector $p_n$ is an idempotent element of $\TL_n$, uniquely
characterized by the following two properties: (1) the coefficient of the
unit element, corresponding to $n$ vertical strands, in the expression for
$p_n$ is $1$ and (2) $p_n$ is ``killed by turnbacks'', that is $p_n D=D
p_n=0$ where $D$ is any planar diagram generator of $\TL_n$ other than the
unit element.

Note that unlike the Jones polynomial and various other link invariants that
have been previously categorified, the coefficients in the expansion of
$p_n$ in terms of the generators of $\TL_n$ are {\em rational}, rather than
polynomial, functions of $q, q^{-1}$. This suggests that categorification of
the projectors cannot be achieved by chain complexes of finite length.

We use Bar-Natan's formulation of Khovanov's theory: the objects in this
category are the Temperley-Lieb diagrams and morphisms are surface
cobordisms in $3$-space between such diagrams, see \cite{MR2174270} and
section \ref{TL section} below.  In this framework, for each $n$ we
construct a chain complex ${P}_n$ whose graded Euler characteristic is the
formal power series corresponding to $p_n$. For example, the power series
for $n=2$ is
\begin{equation} \label{p2}
\MPic{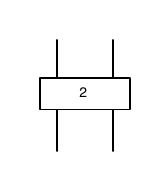} = \MPic{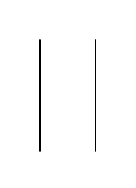} \hspace{-.14in}- \frac{1}{q+q^{-1}}\MPic{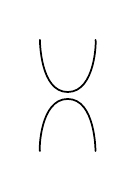}\hspace{-.05in} = \MPic{n2-1} \hspace{-.14in}+ \sum_{i=1}^{\infty}(-1)^i q^{2i-1}\MPic{n2-s}
\end{equation}
We show that the chain complexes ${P}_n$ are uniquely characterized up to
homotopy by properties analogous to those of the Jones-Wenzl projectors
$p_n\in \TL_n$: (1) the identity diagram appears in the chain complex
${P}_n$ only once, in degree zero and (2) ${P}_n$ is contractible ``under
turnbacks'', see the definition of a {\em universal projector} and theorem
\ref{main theorem} in section \ref{main theorem section}. It follows from
these properties that ${P}_n$ is a ``homotopy idempotent'':
${P}_n\otimes{P}_n\simeq {P}_n$. We write down the chain complexes
explicitly for $n=2,3$, see section \ref{formulasec}.  The main technical
part of the paper is the inductive construction of the chain complex ${P}_n$
for larger $n$ in section \ref{proof section}, modeled on the
Frenkel-Khovanov recursion \cite{MR1446615} for the Jones-Wenzl projectors.
The universality properties satisfied by ${P}_n$ and the invariance under
Reidemeister moves, discussed further below, suggest the naturality of the
construction proposed in this paper. We summarize the discussion so far with
the table below.

{\center{\begin{tabular}   {@{\extracolsep{20pt}}ll}
\underline{Algebra}. & \underline{Category}. \\
Temperley - Lieb algebra: $\TL_n$ & Bar-Natan - Khovanov Category: $\Kom(n)$ \\
$p_n \in \TL_n$ & $P_n \in \Kom(n)$, $K_0(P_n)=p_n$ \\
$p_n \cdot p_n = p_n$ &  $P_n\otimes P_n \simeq P_n$ \\
$p_n$ is unique & $P_n$ is unique up to homotopy \\
\end{tabular}}}

An immediate consequence of our construction is a categorification of
quantum spin networks. That is, to a spin network $G$ we associate a chain
complex whose graded Euler characteristic is a Laurent series in $q$
corresponding to the quantum evaluation of $G$.  Some interesting phenomena
are observed here. In the simplest example the rational homology of the
trace of the second projector, $\tr({P}_2)$ has the expected graded Euler
characteristic $[3]=q^{-2}+1+q^2$, but the homology itself has infinite rank
(with extra generators canceling in pairs in the Euler characteristic).
Further, there is $2$-torsion when the homology is taken with integer
coefficients, see \ref{homology subsection}. In section \ref{6j section} we
formulate a categorified analogue of the $6j$-symbols. It takes the form of
an iterated cone construction, giving rise to a ``homotopy change of basis''
in the category of chain complexes.

Our construction also gives rise to an invariant of tangles, leading to a
categorification of the colored Jones polynomial, see section
\ref{Reidemeister section}. Note that the included computations imply that
our work is different from the previously defined categorification of the
colored Jones polynomial \cite{MR2124557} (see also \cite{MR2462446}).  See
\ref{second projector section} for further discussion.

We would like to mention that while preparing this manuscript for
publication, during the MSRI workshop ``Homology Theories of Knots and
Links'' in March 2010 we learned that an alternative,
representation-theoretic, approach to categorifying the Jones-Wenzl
projectors has been pursued by Igor Frenkel, Catharina Stroppel and Joshua
Sussan \cite{FSS}. In light of the universality properties of our
construction (see section 3), it is plausible that the two approaches are
equivalent, although the methods are quite different.  One advantage of
working in Khovanov's and Bar-Natan's framework for categorification of the
Temperley-Lieb algebra is that our construction of the categorified
projectors is explicit and it is readily available for topological
applications. The interested reader may want to compare our construction in
section \ref{6j section} to the discussion of the $6j$-symbols in
\cite[Section 17]{FSS}.

We would like to add that more recently Lev Rozansky \cite{Roz} has proposed
an elegant idea on categorification of the Jones-Wenzl projectors, based on
the properties of the infinite torus braid. Our construction is based on the
Frenkel-Khovanov recursive formula, however it seems reasonable to believe
that the two approaches may be related (and more generally the universality
properties satisfied by the projectors imply that the different
constructions are homotopy equivalent).

{\bf Acknowledgements}. The authors would like to thank the referee for a
number of suggestions which have led to a better exposition.

\section{The Temperley-Lieb Algebra and the Jones-Wenzl Projectors} \label{TL section}

This section summarizes the relevant background on the definition and
categorification of the Temperley-Lieb algebra. Section \ref{homotopy
  lemmas} states a version of the Gaussian elimination lemma which will be
used throughout the paper.

\subsection{Temperley-Lieb Algebra} The Temperley-Lieb algebra is the unital
$\mathbb{Z}[q,q^{-1}]$-algebra of $\quantsu$-equivariant maps between
$n$-fold tensor powers of the fundamental representation $V$,
$$ \TL_n = \Morph_{\quantsu}(V^{\otimes n},V^{\otimes n}). $$

There is an explicit presentation given by the standard generators $1$ and
$e_i$, $0 < i < n$, satisfying the relations:

\begin{enumerate}
\item $e_i e_j = e_j e_i$ if $|i-j| \geq 2$.
\item $e_i e_{i\pm 1} e_i = e_i$
\item $e_i^2 = [2] e_i$
\end{enumerate}

where the quantum integer $[n]$ is defined to be

$$[n] = \frac{q^n - q^{-n}}{q-q^{-1}} = q^{-(n-1)} + q^{-(n-3)} +
\cdots +  q^{n-3} + q^{n-1}$$

Each generator $e_i$ can be pictured as a diagram consisting of $n$ chords
between two collections of $n$ points on two horizontal lines in the plane.
All strands are vertical except for two, connecting the $i$th and the
$(i+1)$-st points in each collection. For instance, when $n = 3$ we have the
following diagrams,

$$1 = \CPic{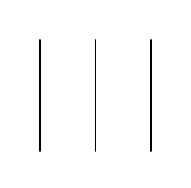}\,\,,\quad e_1 = \CPic{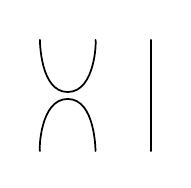} \textnormal{\quad  and\quad   } e_2 = \CPic{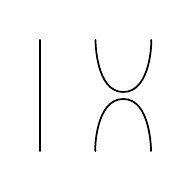}. $$

The multiplication is given by vertical composition of diagrams, and planar
isotopy induces relations (1) and (2) between the generators above. The third
relation says that any circles which are created may be removed at the cost
of multiplication by $[2]= q+q^{-1}$.

This algebra is well-known in low-dimensional topology in particular due to
its natural extension from planar diagrams to tangles, captured by the
Kauffman bracket relations:

$$\begin{diagram} \CPic{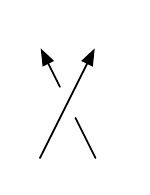} \hspace{-.3in} & = & q \CPic{n2-1}\hspace{-.15in} & - & \hspace{.2in} q^{2} \CPic{n2-s}\hspace{-.15in} & \normaltext{\quad\,\, and } &
\CPic{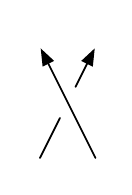}\hspace{-.3in} & = & q^{-2} \CPic{n2-s} & - & \hspace{.2in} q^{-1}\CPic{n2-1}
\end{diagram}
$$

which yield the Jones Polynomial up to normalization \cite{MR1280463, MR1622916}.

\subsection{Jones-Wenzl Projectors} \label{JW projectors} The Jones-Wenzl Projectors $p_n
\in \TL_n$ are idempotent elements of the Temperley-Lieb algebra which have
proven to be fundamental to its study and applications. The projectors
appear in the study of spin networks or the graphical calculus of higher
$\quantsu$ representations, the colored Jones polynomial and many
constructions of Chern-Simons theory
\cite{MR1280463,MR1191386,MR1292673,MR1362791,MR1797619, Walker91}.

The projectors were originally \cite{MR873400} defined by the recurrence
relation,
\begin{align*}
p_1 &= 1\\
p_n &= p_{n-1} - \frac{[n-1]}{[n]} p_{n-1} e_{n-1} p_{n-1}
\end{align*}

If we depict $p_n$ graphically by a box with $n$ incoming and outgoing
chords:
$$p_n = \BPic{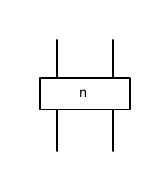}$$
then the formula may be illustrated as follows:
$$\BPic{pn} = \BPic{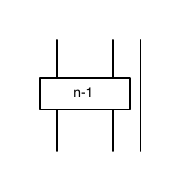} -\quad \frac{[n-1]}{[n]}\! \BPic{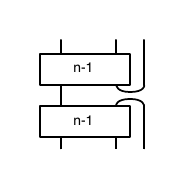}$$

It can be shown that the Jones-Wenzl projectors are uniquely characterized by
the following properties:

\begin{enumerate}
\item $p_n \in \TL_n$ considered as a $\mathbb{Z}[q^{-1}]\llbracket q \rrbracket$-algebra.
\item $p_n - 1$ belongs to the subalgebra generated by
  $\{e_1,e_2,\ldots,e_{n-1}\}$
\item $e_i p_n = p_n e_i = 0$ for all $i = 1, \ldots, n-1$.
\end{enumerate}

See for instance \cite{MR1280463, MR1472978}.
The coefficients of Temperley-Lieb diagrams in the expression for $p_n$
uniquely determine power series with positive powers of $q$. The equations
above then define $p_n$ as a power series in Temperley-Lieb elements, for
example the power series for $p_2$ is given in (\ref{p2}) in the
introduction. (Alternatively, one could expand rational functions as series
in $q^{-1}$ producing a dual projector, see discussion following definition
\ref{universal projector def}.)

\subsection{Categorification of the Temperley-Lieb algebra} \label{categorified TL section}

Work by a number of authors on the existence of integral bases in Lie theory
led to a categorification of the Temperley-Lieb algebra by Mikhail Khovanov
in which integer coefficients were interpreted as the dimensions of graded
vector spaces and polynomials as graded Euler characteristics
\cite{MR1446615, MR1714141, MR1740682, MR1928174, MR2120117, MR2305608}. This
construction extends to tangles and there is a corresponding functoriality
with respect to cobordisms between these tangles \cite{MR2113903,MR2174270}.

In this section we recall Dror Bar-Natan's graphical formulation
\cite{MR2174270} of the Khovanov categorification. It will be used
throughout the remainder of this paper. Using the Bar-Natan formulation has
the advantage of allowing our constructions to apply to a number of variant
categorifications which exist in the literature.

\begin{definition}\label{groth}
If $\mathcal{C}$ is an additive category then the \emph{split Grothendieck
  group} of $\mathcal{C}$ is
$$K_0(\mathcal{C}) = \mathbb{Z}\inp{\Iso(\mathcal{C})} / ([A \oplus B] = [A] + [B]),$$

the free abelian group generated by isomorphism classes of objects in
$\mathcal{C}$ modulo the relation above. If $\mathcal{C}$ is a monoidal
category then the map $\mathcal{C} \otimes \mathcal{C} \to \mathcal{C}$
induces $K_0(\mathcal{C})\otimes K_0(\mathcal{C}) \to K_0(\mathcal{C})$,
endowing $K_0(\mathcal{C})$ with an algebra structure. See \cite{WeibelK, MR2559652}.
\end{definition}

Our goal is to define an additive monoidal category $\Cob(n)$ such that
$K_0(\Cob(n)) \cong \TL_n$. There is an additive category $\PCob(n)$ whose
objects are isotopy classes of formally $q$-graded Temperley-Lieb diagrams
with $2n$ boundary points. The morphisms are given by the free
$\mathbb{Z}$-module spanned by isotopy classes of orientable cobordisms
bounded in $\mathbb{R}^3$ between any two planes containing such diagrams.

The \emph{degree} of a cobordism $C : q^i A \to q^j B$ is given by

$$\deg(C) = \deg_t(C) + \deg_q(C)$$

where the topological degree $\deg_t(C) = \chi(C) - n$ is given by the Euler
characteristic of $C$ and the $q$-degree $\deg_q(C) = j - i$ is given by the
relative difference in $q$-gradings.  The maps $C$ used throughout the paper
will satisfy $\deg(C) = 0$. The formal $q$-grading will be chosen to cancel
the topological grading. When working with chain complexes every object will
also contain a homological grading and every map will have an associated
homological degree. Homological degree is not part of the definition
$\deg(C)$.

It has become a common notational shorthand to represent a handle by a dot
and a saddle by a flattened diagram containing a dark line.
$$\BPic{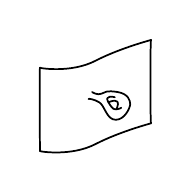}\, =\! 2\!\! \BPic{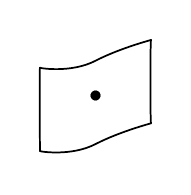} = 2\!\! \BPic{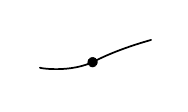} \textnormal{ and } \BPic{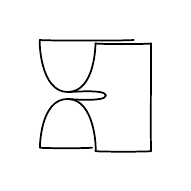}\,\, =\!\! \BPic{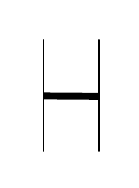}$$

(The topological degrees of the cobordisms above are $-2$, $-1$ respectively.)
We would like a category $\mathcal{C}$ such that $\Ktheory_0(\mathcal{C})
\cong \TL_n$ so we require that the object represented by a closed circle be
isomorphic to sum of two empty objects in degrees $\pm 1$ respectively. If
such maps are to be degree preserving then the most natural choice for these
maps is given below.
$$\begin{diagram}
\varphi\!: \hspace{-.15in}
&\CPic{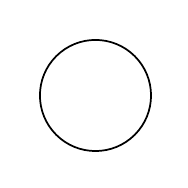}\hspace{.1in} & \pile{\rTo^{\left( \CPic{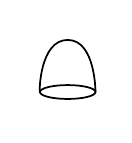} \CPic{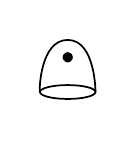} \right)  }\\
\lTo_{\left( \CPic{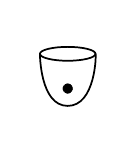} \CPic{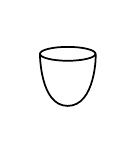} \right)  }}
& \hspace{.1in} q^{-1}\, \emptyset\,\, \oplus\,\, q\, \emptyset &\hspace{.15in} : \! \psi
\end{diagram}
$$

In order to obtain $\varphi \circ \psi = 1$ and $\psi \circ \varphi = 1$ we
form a new category $\Cob(n) = \Cob^3_{\cdot/l}(n)$ obtained as a quotient
of the category $\PCob(n)$ by the relations given below.
 $$\CPic{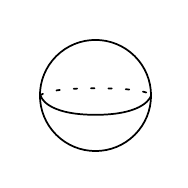} = 0 \hspace{.75in} \CPic{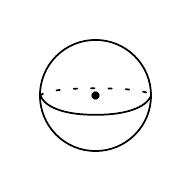} = 1 \hspace{.75in} \CPic{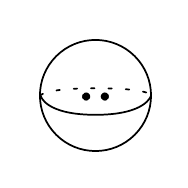} = 0 \hspace{.75in} \CPic{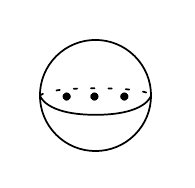} = \alpha$$
 $$\CPic{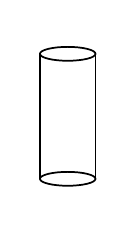} = \CPic{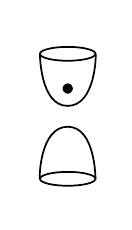} + \CPic{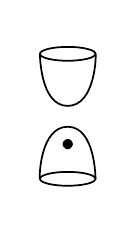}$$

The cylinder or neck cutting relation implies that closed surfaces
$\Sigma_g$ of genus $g > 3$ evaluates to $0$. In what follows we will
let $\alpha$ be a free variable and absorb it into our base ring ($\Sigma_3=
8\alpha$). One can think of $\alpha$ as a deformation parameter, see \cite{MR2174270} for further details.

\begin{example}{(The Circle)}\label{thecircle}
In both $\PCob(0)$ and $\Cob(0)$ there are objects associated to the circle
and the empty set. Consider
$$A = \Morph_{\PCob(0)}(\emptyset, \MPic{circle}\!\!) \normaltext{\quad and \quad} B = \Morph_{\Cob(0)}(\emptyset, \MPic{circle}\!\!).$$

Both $A$ and $B$ are abelian groups. An element of $A$ consists of a linear
combination of isotopy classes of orientable surfaces with a single fixed
boundary circle. An element of $B$ consists of a linear combination of such
surfaces subject to the relations above. In particular, the last relation
allows us to cut any surface along a closed curve which bounds a disk in
3-space. The reader can check that every element $x \in B$ is of the form:
$$ x = m \CPic{torus-null-to-circ}\! \! + n \CPic{cap-null-to-circ}$$

for some $m, n$.
$\Morph_{\Cob(0)}(\emptyset, \MPic{circle}\!\!)$ can be endowed with a Frobenius algebra structure using maps induced by cobordisms.
When ${\alpha}=0$ this is the Frobenius algebra which appears in Khovanov's original construction, see \cite{MR2232858}.
\end{example}

Given two objects $C,D \in \Cob(n)$ we will use $C\otimes D$ to denote the
map $\Cob(n) \otimes
\Cob(n) \to \Cob(n)$ obtained by gluing all diagrams and morphisms along the
$n$ boundary points and $n$ boundary intervals respectively. Pictorially,
$$C\otimes D = \BPic{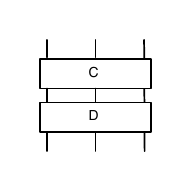}$$

\begin{lemma}
There is an isomorphism of $\mathbb{Z}[q,q^{-1}]$-algebras,

$$K_0(\Cob(n)) \cong \TL_n$$
\end{lemma}
\begin{proof}
Note that the $q$-degree shifting functor determines an endomorphism
$K_0(q) : K_0(\Cob(n)) \to K_0(\Cob(n))$ making $K_0(\Cob(n))$ into a
$\mathbb{Z}[q,q^{-1}]$-algebra. The proof follows directly from the
construction above.
\end{proof}

The categorifications $\Cob(n)$ fit together in much the same way as the
Temperley-Lieb algebras $\TL_n$. There is an inclusion $ - \sqcup 1^{m-n} :
\Cob(n) \to \Cob(m)$ whenever $n\leq m$ obtained by unioning each diagram
with $m-n$ disjoint vertical line segments on the right to each object and $m-n$ disjoint
disks to each morphism. If $m=n$ then the empty set is used instead of
either intervals or disks.

\begin{definition}\label{kom def}
  Let $\Kom(n) = \Kom(\Mat(\Cob^3_{\cdot/l}(n)))$ be the category of chain
  complexes of finite direct sums of objects in $\Cob^3_{\cdot/l}(n)$.  We
  allow chain complexes $K_*$ of unbounded positive homological degree and
  require that for each $K_*$ there exists an $N \in \mathbb{Z}_-$ such that $K_n = 0$ for
  $n < N$.
\end{definition}

Note that $\otimes$ extends to $\Kom(n)$, see \cite{gelfandmanin}. The
functor $- \sqcup 1^{m-n}$ extends to $\Kom(n)$ in the obvious way. The
skein relation becomes
$$\begin{diagram} \CPic{n2-orcross} \hspace{-.3in} & = & q \CPic{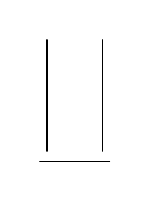}\hspace{-.15in} &\rTo^{\!\MPic{n2-1-sad}} & \hspace{.2in} q^{2} \CPic{n2-s}\hspace{-.15in} & \normaltext{\quad\,\, and } &
\CPic{n2-orcross-2}\hspace{-.3in} & = & q^{-2} \CPic{n2-s} & \rTo^{\!\MPic{n2-1-sad}} & \hspace{.2in} q^{-1}\CPic{n2-1-u}
\end{diagram}
$$

where the underlined diagram represents homological degree 0.

The skein relation allows us to associate to any tangle diagram $D$ with
$2n$ boundary points an object in $\Kom(n)$.

\subsubsection{Grothendieck group of $\Kom(n)$}\label{catdiscuss}

We have included this section in order to explain how the word
categorification pertains to the categories of partially unbounded chain
complexes appearing in this paper. What appears here represents only a minor
variation of the theory, see \cite{WeibelK}.

\begin{definition}
If $K$ is a finite sum of $q$-graded objects then define $\max_q(K)$ to be
the maximum $q$-degree of a summand of $K$. Define $\min_q(K)$ to be the
minimal $q$-degree of a summand of $K$.

Let $\Kom^0(n) \subset \Kom(n)$ be the subcategory consisting of chain
complexes $K_*$ such that the functions $f(i) = \max_q K_i$ and $g(j) =
\min_q K_j$ are both monotonically increasing.
\end{definition}

All of the constructions in this paper reside within this subcategory of
$\Kom(n)$. For example, the diagram in section \ref{fk sequence section},
which is important in the proof of the main theorem, has monotonically
increasing $q$-degree.  It follows from the definition of $\Kom^0(n)$ that
the $q$-graded Euler characteristic of any object $C \in \Kom^0(n)$ is
well-defined as a power series in $\mathbb{Z}[q^{-1}]\llbracket q
\rrbracket$.

\begin{definition}
The \emph{Grothendieck group} of $\Kom$ is
$$K_0(\Kom) = \mathbb{Z}\inp{\Iso(\Kom^0)} \Big/ \left([K_*] = \sum_{i=-\infty}^{\infty} (-1)^i [K_i]\right),$$

the free abelian group generated by isomorphism classes of objects in
$\Kom^0$ modulo the relation above. We again identify $[q^j K_i] = q^j
[K_i]$, and the monotonicity implies that these relations are well-defined
power series. Since the objects of $\Kom$ are bounded in negative
homological degree the sum above contains only finitely many terms with $i <
0$.  $K_0(\Kom)$ inherits an algebra structure as in definition \ref{groth}.
\end{definition}

\begin{lemma}
There is an isomorphism of $\mathbb{Z}[q^{-1}]\llbracket q \rrbracket$-algebras
$$K_0(\Kom(n)) \cong \TL_n.$$
\end{lemma}

The proof is standard.

\begin{remark} The monotonicity condition imposed on
$\Kom(n)$ is very natural in our context.
Let $f : A \to B$ be a morphism in $\Cob(n)$, then one of the following holds:
\begin{enumerate}
\item $\deg_t(f) > 0$ and $f = 0$
\item $\deg_t(f) = 0$, $A=B$, and  $f = m 1_A$, for some $m \in \mathbb{Z}$.
\item $\deg_t(f) < 0$.
\end{enumerate}
\end{remark}

This follows from the classification of surfaces and the relations imposed
on surfaces in $\Cob(n)$. For instance, the only surface of positive Euler
characteristic is the sphere which evaluates to 0. We conclude that the
graded components of non-trivial chain complexes $K_*$ in $\Kom(n)$ with
degree 0 differential, as in definition \ref{universal projector def} below,
are monotonically increasing in the sense defined above.

\subsection{Chain Homotopy Lemmas} \label{homotopy lemmas}

We will make frequent use of the following standard lemma in this paper, see \cite{MR2174270,MR2457839}. For the definition of homotopy see \cite{gelfandmanin}.

\newcommand{\clxx}[2]{\ensuremath{\begin{array}{c} #1 \\ \oplus \\ #2 \end{array}}}
\newcommand{\matot}[2]{\ensuremath{\left(\begin{array}{c} #1 \\ #2 \end{array}\right)}}
\newcommand{\matto}[2]{\ensuremath{\left(\begin{array}{cc} #1 & #2 \end{array}\right)}}
\newcommand{\mattt}[4]{\ensuremath{\left(\begin{array}{cc} #1 & #2\\ #3 & #4 \end{array}\right)}}

\newcommand{\clxxx}[3]{\ensuremath{
\begin{array}{c} #1 \\ \oplus \\ #2 \\ \oplus \\ #3 \end{array}}}
\newcommand{\clxy}[2]{\ensuremath{
\begin{array}{c} #1 \\ \oplus \\ #2  \end{array}}}

\newcommand{\clxxABC}[1]{ \clxxx{A_#1}{B_#1}{C_#1}}
\newcommand{\clxyAC}[1]{ \clxy{A_#1}{C_#1}}

\begin{lemma}{(Gaussian Elimination)} \label{gaussian elimination}
  Let $K_*$ be a chain complex in an additive category $\mathcal A$
  containing a subcomplex isomorphic to the top row below. If $\varphi : B
  \to D$ is an isomorphism there is a homotopy equivalence from $K_*$ to a
  smaller complex containing the bottom row below.

$$\begin{diagram}
A &\rTo^{\matot{\cdot}{\alpha}} & \clxx{B}{C} & \rTo^{\mattt{\varphi}{\lambda}{\mu}{\eta}} & \clxx{D}{E} & \rTo^{\matto{\cdot}{\epsilon}} & F\\
\dTo & & \dTo & & \dTo & & \dTo   \\
A &\rTo^\alpha & C & \rTo^{\eta - \mu\varphi^{-1}\lambda} & E & \rTo^{\epsilon} & F
\end{diagram}$$
\end{lemma}

\smallskip

\begin{remark}\label{deloop}
Often a preliminary step before an application of Gaussian Elimination will
be delooping.  This consists of using the isomorphisms $\varphi$ and $\psi$
defined in section \ref{categorified TL section} to remove a disjoint circle
from a diagram and adjust any maps to and from the diagram accordingly. In
applications this results in an isomorphism which can be removed using lemma
\ref{gaussian elimination}. For example, see the proof of theorem
\ref{secondproj thm}.
\end{remark}

The following result is a direct generalization which will be very useful in our context.

\begin{lemma}{(Simultaneous Gaussian Elimination)} \label{sim gaussian elimination}
Let $K_*$ be a chain complex in an additive category $\mathcal A$ of the form

$$K_* = \begin{diagram}
\clxyAC{0}  & \rTo^{M_0}  &  \clxxABC{1} & \rTo^{M_1} & \clxxABC{2} & \rTo^{M_2} & \clxxABC{3} & \rTo^{M_3} & \cdots
\end{diagram}$$
where

$$M_0 =
\left(\begin{array}{cc}
a_0 & c_0\\
d_0 & f_0\\
g_0 & j_0
\end{array}\right) \quad \textnormal{   and   } \quad M_i =
\left(\begin{array}{ccc}
a_i & b_i & c_i\\
d_i & e_i & f_i\\
g_i & h_i & j_i
\end{array}\right) \textnormal{ for all $i > 0$ }
$$

If $a_{2i} : A_{2i} \to A_{2i+1}$ and $e_{2i+1} : B_{2i+1} \to B_{2i+2}$ are
isomorphisms for $i \geq 0$ then the chain complex $K_*$ is homotopy
equivalent to the smaller chain complex $D_*$ obtained by removing all $A_i$
and $B_i$ terms via the isomorphisms $a_{2i}$ and $e_{2i+1}$:

$$D_* = \begin{diagram}
C_0 & \rTo^{q_0} & C_1 & \rTo^{q_1} & C_2 &\rTo^{q_2} & C_3 & \rTo^{q_3} & \cdots
\end{diagram}
$$

where $q_{2i} = j_{2i} - g_{2i} a_{2i}^{-1} c_{2i}$ and $q_{2i+1} = j_{2i+1} - h_{2i+1} e_{2i+1}^{-1} f_{2i+1}$.

\end{lemma}

\begin{proof}

  First apply Gaussian elimination to each isomorphism $a_{2i}$ in order to
  obtain the chain complex

$$\begin{diagram}
C_0 & \rTo^{X} & \clxy{B_1}{C_1} & \rTo^{Y_1} & \clxy{B_2}{C_2} & \rTo^{Y_2} & \clxy{B_3}{C_3} & \rTo^{Y_3} & \clxy{B_4}{C_4} & \rTo & \cdots
\end{diagram}$$

where $X = \matot{f_0 - c_0 a_0^{-1} d_0}{j_0 - c_0 a_0^{-1} g_0}$\quad and

\begin{align*}
Y_{2i} &= \mattt{e_{2i}-d_{2i} a_{2i}^{-1} b_{2i}}{f_{2i}-d_{2i}a_{2i}^{-1} c_{2i}}
{h_{2i}-g_{2i}a_{2i}^{-1} b_{2i}}{j_{2i}-g_{2i}a_{2i}^{-1} c_{2i}} &
Y_{2i+1} &= \mattt{e_{2i+1}}{f_{2i+1}}{h_{2i+1}}{j_{2i+1}}
\end{align*}

Now apply Gaussian elimination to each isomorphism $e_{2i+1}$ in order to
obtain the chain complex $D_*$ above.
\end{proof}

A chain complex $K_*$ is contractible if $K_* \simeq 0$, see
\cite{gelfandmanin}.

\begin{lemma}{(Big Collapse)} \label{big collapse}
  A chain complex $K_*$ of contractible chain complexes $K_i$ is
  contractible.
\end{lemma}

\section{Universal Projectors and Statement of the Main Theorem} \label{main theorem section}
The projectors defined in this paper satisfy a universal property making
them unique up to homotopy.

\begin{definition} \label{universal projector def}
  A chain complex $(P_*,d_*) \in \Kom(n)$ is a \emph{universal projector} if
\begin{enumerate}

\item It is positively graded with degree zero differential.
\begin{enumerate}
\item $P_k = 0$ for all $k < 0$ and $\deg_q(P_k) \geq 0$ for all $k > 0$.
\item $d_k$ is a matrix of degree zero maps for all $k \in \mathbb{Z}$.
\end{enumerate}

\item The identity diagram appears only in homological degree zero and only
  once.
\begin{enumerate}
\item $P_0 \cong 1$
\item $P_k \not\cong 1 \oplus D$ for any $D \in \Mat(\Cob(n))$ for all $k
  > 0$.
\end{enumerate}

\item The chain complex $P_*$ is contractible ``under turnbacks,'' that is for any generator $e_i \in \TL_n$, $0<i<n$,
\begin{enumerate}
\item $P_* \otimes e_i \simeq 0$
\item $e_i \otimes P_* \simeq 0$
\end{enumerate}
\end{enumerate}
\end{definition}

See section \ref{categorified TL section} for a discussion of degrees.

Compare these axioms to the axioms in section \ref{JW projectors}
characterizing the Jones-Wenzl projectors $p_n\in\TL_n$. The first two
axioms are non-triviality conditions. The first excludes uninteresting
variants of the definition obtained by degree shifting and symmetry. For
instance, we could change (1) to require a negative $q$-grading and reverse
of all the arrows. The second excludes contractible complexes from
consideration. The third axiom implies that composing the projector with any
Temperley-Lieb diagram which is not identity yields an object in $\Kom(n)$
which is homotopic to the zero complex.

If $P_n \in \Kom(n)$ is a universal projector then $[P_n] \in
K_0(\Kom(n)) \cong \TL_n$ satisfies the axioms in section \ref{JW
  projectors}, implying that $[P_n] = p_n \in \TL_n$. See section
\ref{catdiscuss} for a detailed discussion of what categorification means in
this context.

In \cite{MR2232858} it was shown that those categorifications of $\TL_n$
giving rise to functorial invariants of tangles must be categories
containing Frobenius algebra objects which are quotients of the one
described by section \ref{categorified TL section} (see example
\ref{thecircle}). Since the construction of the universal projector only
uses relations which follow from the existence of such an algebra it is
possible that the construction of the universal projector in $\Kom(n)$ could
be carried out within other categorifications.

It is important to note that our definition \emph{disagrees} with some
previous categorifications based on different axiomizations of the
Jones-Wenzl projectors such as the dimension axiom \cite{MR2124557} (for
related work see \cite{MR2462446, Grigsby}):
$$\tr(p_n) = [n+1]$$

This is implied by the homotopy uniqueness corollary below and the
computation of $\Homology_*(\tr(P_2))$ contained in the next section.

We can now state the main theorem of the paper.

\begin{theorem} \label{main theorem} For each $n > 0$, there exists a chain
  complex $C \in \Kom(n)$ that is a universal projector.
\end{theorem}

We summarize some immediate consequences of the axioms in definition
\ref{universal projector def}. See also sections \ref{Reidemeister section}
and \ref{spin networks section}.

\begin{proposition} \label{composition of projectors}
  If $C \in \Kom(n)$ is a universal projector and $D \in \Kom(m)$ is a
  universal projector such that $0 \leq m \leq n$ then
$$C \otimes (D \sqcup 1^{n-m}) \simeq C \simeq (D \sqcup 1^{n-m} ) \otimes C$$

  Pictorially,
$$\BPic{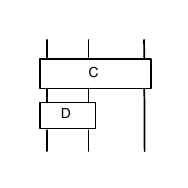} \simeq\!\! \BPic{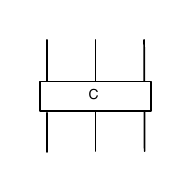} \simeq\!\! \BPic{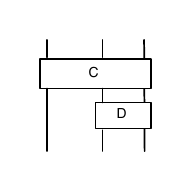}$$

\end{proposition}

\begin{proof}
  The tensor product of chain complexes $C_* \otimes (D_* \sqcup 1^{n-m})$
  is the total complex of a bicomplex which can be written as a chain
  complex of chain complexes:

$$C_* \otimes (D_0 \sqcup 1^{n-m}) \to C_* \otimes (D_1 \sqcup 1^{n-m}) \to C_* \otimes (D_2 \sqcup 1^{n-m}) \to \cdots $$

Or graphically,
$$\begin{diagram}
\BPic{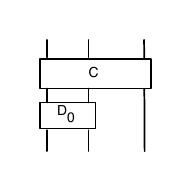} & \rTo & \!\!\BPic{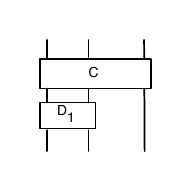} & \rTo & \!\!\BPic{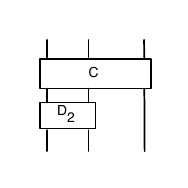} & \rTo & \cdots
\end{diagram}$$

By the second axiom $D_0 = 1$ in homological degree $0$ and so the identity
diagram cannot be found as a summand of $D_k \sqcup 1^{n-m}$ for any $k >
1$. In addition, $C$ satisfies axiom 3 so it follows that:

\begin{enumerate}
\item The homological degree $0$ portion of this complex is isomorphic to
$C_*$
\item All chain complexes in degree above zero are contractible.
\end{enumerate}

Lemma \ref{big collapse} (big collapse) implies that there is a homotopy
equivalence $C_* \simeq C_* \otimes (D_* \sqcup 1^{n-m})$. The other
equivalence $(D_* \sqcup 1^{n-m} ) \otimes C_* \simeq C_*$ is proven in the
same manner.
\end{proof}

A special case of the above proposition with $m=n$ implies that universal
projectors we have defined behave like idempotent elements. In other words,
if $C$ is a universal projector then the functors $C\otimes -$ and $-\otimes
C$ are idempotent on the homotopy category of $\Kom(n)$.

\begin{corollary}{(Idempotence)} \label{idempotency cor}
If $C \in \Kom(n)$ is a universal projector then

$$C \otimes C \simeq C$$

This is represented diagrammatically as
$$\BPic{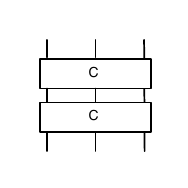} \simeq \!\!\BPic{compsquare2}$$

\end{corollary}

Proposition \ref{composition of projectors} also implies that the universal projectors are unique
up to homotopy.

\begin{corollary}{(Homotopy Uniqueness)} \label{homotopy uniqueness corollary}
  If $C, D \in \Kom(n)$ are universal projectors then $C \simeq D$.

\end{corollary}
\begin{proof}
  Let $C,D \in \Kom(n)$ be universal projectors. Proposition \ref{composition of projectors}  holds when $n=m$ so that $1^0 = \emptyset$,

$$ C \simeq C \otimes (D \sqcup \emptyset) \cong (C \sqcup \emptyset) \otimes D \simeq D$$

In pictures,
$$\BPic{compsquare2} \simeq \!\!\BPic{compsquareCD} \simeq \!\!\BPic{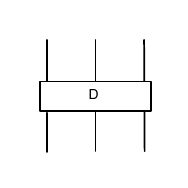}$$
\end{proof}

\section{Explicit Formulae and Computations} \label{formulasec}

We now give some explicit examples of lower order projectors. The second
projector below will play a role in the proof of the main theorem. Higher
projectors are much more complicated.

\subsection{The Second Projector} \label{second projector section}

The second projector is defined to be the chain complex

$$\begin{diagram}
\CPic{p2box}\! & = & \!\CPic{n2-1} & \rTo^{\MPic{n2-1-sad}} &q \CPic{n2-s} &\rTo^{\MPic{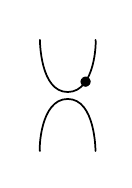} \!\!\!- \MPic{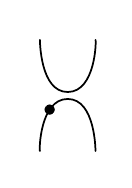}} &q^{3} \CPic{n2-s} &
\rTo^{\MPic{n2-tops}\!\!\! + \MPic{n2-bots}} & q^{5} \CPic{n2-s} &  \cdots
\end{diagram}$$

in which the last two maps alternate ad infinitum. More explicitly,

$$P_2 = (C_*, d_*)$$

The chain groups are given by
$$
C_n =
\left\{
\begin{array}{lr}
q^0 \MPic{n2-1}&n = 0\\
q^{2n-1} \MPic{n2-s}&n > 0
\end{array}
\right.
$$

The differential is given by

$$
d_n =
\left\{
\begin{array}{llr}
\MPic{n2-1-sad}                 & : \MPic{n2-1}\!\! \to q \MPic{n2-s}   & n = 0\\
\MPic{n2-tops} \!\!\! + \MPic{n2-bots} & :   q^{4k-1} \MPic{n2-s}\!\!  \to q^{4k +1} \MPic{n2-s}                    &n \ne 0, n = 2k\\
\MPic{n2-tops} \!\!\! - \MPic{n2-bots} & : q^{4k +1} \MPic{n2-s}\!\!  \to q^{4k+3} \MPic{n2-s}                    & n = 2k+1
\end{array}
\right.
$$

\begin{proposition}
$P_2$ defined above is a chain complex, that is successive compositions of the differential are equal to zero.
\end{proposition}
\begin{proof}
Since $d_{2n+1} \circ d_{2n} = d_{2n} \circ d_{2n-1}$ there are only two cases,
\begin{eqnarray*}
d_1\circ d_0 &= & \MPic{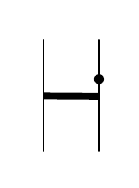}\!\!\! - \MPic{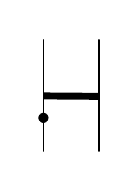} \\
&= & \MPic{n2-s-bots} \!\!\!- \MPic{n2-s-bots} =  0\\
\end{eqnarray*}
The second equality follows from moving the dot from one side of the saddle
to the other, and isotopic cobordisms are considered equal by construction, see section \ref{categorified TL section}.

\begin{eqnarray*}
 d_{2n+1} \circ d_{2n} &= & (\MPic{n2-tops} \!\!\! + \MPic{n2-bots})\circ (\MPic{n2-tops}\!\!\! - \MPic{n2-bots}) \\
&=& \MPic{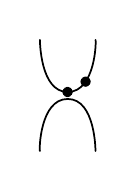} \!\!\! + \MPic{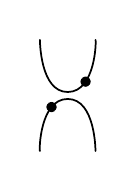}\!\!\! - \MPic{n2-mid-2} \!\!\! - \MPic{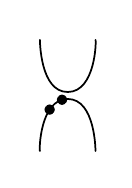} \\
&=& \alpha \MPic{n2-s} \!\!\! + 0 - \alpha \MPic{n2-s} = 0\\
\end{eqnarray*}

The last equality follows because the relations of section \ref{categorified TL section} allow us to replace two dots by multiplication with $\alpha$.

\end{proof}

\begin{theorem}\label{secondproj thm}
  The chain complex $P_2 \in \Kom(2)$ defined above is a universal projector.
\end{theorem}
\begin{proof}
  Since the identity object only appears in degree $0$ and the chain complex
  is positively graded with degree zero differentials, axioms 1 and 2 are
  satisfied by definition.  For axiom 3, note that there is only one
  standard generator $e_1 \in \TL_2$ and the vertical symmetry in the
  definition of $P_2$ implies $P_2\otimes e_1 \cong e_1\otimes P_2$. Consider $e_1 \otimes P_2$:

$$\begin{diagram}
\CPic{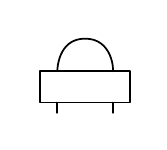} \!\!\! & = & \!\!\! \CPic{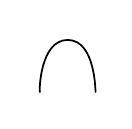} &
\rTo^{\MPic{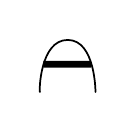}} &
q \CPic{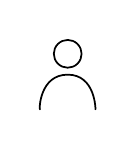} &
\rTo^{\MPic{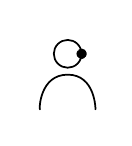} \!\!\!- \MPic{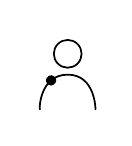}} &
q^{3} \CPic{n2-capbub} &
\rTo^{\MPic{n2-capbub-tops} \!\!\! + \MPic{n2-capbub-bots}} &
q^{5} \CPic{n2-capbub} &  \!\!\! \!\!\!\cdots
\end{diagram}$$

The top strand of $e_1$ has been omitted in the illustration above.  Now we ``deloop'' (see
remark in section \ref{deloop}) and conjugate our differentials by the isomorphism
$\varphi$ in section \ref{categorified TL section} to obtain the isomorphic
complex

$$\begin{diagram}
\CPic{n2-hs} \!\!\! &
\rTo^{\!\!\! M_0} &
q^{0} \CPic{n2-hs} \!\!\! \oplus\,\, q^{2} \CPic{n2-hs} \!\!\!&
\rTo^{M_1} &
q^{2} \CPic{n2-hs} \!\!\! \oplus\,\, q^{4} \CPic{n2-hs} \!\!\!&
\rTo^{M_2} &
q^{4} \CPic{n2-hs} \!\!\! \oplus\,\, q^{6} \CPic{n2-hs}\!\!\! &  \cdots
\end{diagram}$$

where
$$
M_0 = \left(\begin{array}{c}
\MPic{n2-hs} \\
\MPic{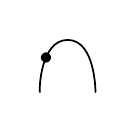}
\end{array} \!\!\!\! \right)  \hspace{.5cm}
M_1 = \left(\begin{array}{cc}
-\MPic{n2-hs-dot} & \MPic{n2-hs}\\
\alpha\MPic{n2-hs} & -\MPic{n2-hs-dot}
\end{array}
\!\!\!\!\right) \hspace{.5cm}
M_2 = \left(\begin{array}{cc}
\MPic{n2-hs-dot} & \MPic{n2-hs}\\
\alpha\MPic{n2-hs} & \MPic{n2-hs-dot}
\end{array}
\!\!\!\!\right)  \\
$$

Applying lemma \ref{sim gaussian elimination} (simultaneous Gaussian
elimination) using $C_* = 0$ and selecting for $a_{2i}$ and $e_{2i+1}$ the
first component of $M_0$ and the identity map in the upper righthand component
of each successive matrix shows that the complex is homotopic to the zero
complex.
\end{proof}

\subsubsection{Homology of the Trace} \label{homology subsection}

In the Temperley-Lieb algebra the trace of any diagram $D \in \TL_n$ is
defined to be the element $\tr(D)\in \mathbb{Z}[q,q^{-1}]$ associated with
the diagram obtained by connecting each of the bottom boundary points to the
corresponding top points by parallel arcs in the plane:
$$\tr(D) = \BPic{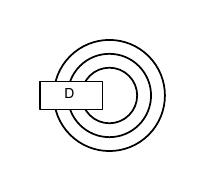}$$

The Jones-Wenzl projectors, $p_n \in \TL_n$ are commonly known to
satisfy

$$\tr(p_n) = [n+1]$$

In fact, they can be characterized by this property together with the
turnback axiom 3 of definition \ref{universal projector def}.  One would
expect then that the graded Euler characteristic of the complex given by the
trace of the universal projectors defined in this document are given by the
polynomials $[n+1]$.

This is true when the coefficient ring is rational and $\alpha = 0$. It is
however \emph{not true} that the homology of $\tr(P_2)$ is spanned
\emph{only} by classes that correspond to coefficients of the graded Euler
characteristic; the homology contains infinitely many terms which cancel in
the graded Euler characteristic.

When $\alpha = 0$ and coefficients are rational,

$$
\Homology_n(\tr(P_2);\Mb{Q}) =
\left\{
\begin{array}{ll}
q^{-2}\Mb{Q} \oplus q^0\Mb{Q} & n = 0\\
0                                   & n = 1\\
q^{4k-2}\Mb{Q}                    &n = 2k \normaltext{ and } k >0\\
q^{4k+2}\Mb{Q}                    & n = 2k+1 \normaltext{ and } k > 0\\
\end{array}
\right.
$$

Note that the graded Euler characteristic equals $[3] = q^{-2} + 1 +
q^{2}$. All other terms cancel in pairs.

If $\alpha = 0$ and the coefficient ring is integral then there is an
additional infinite family of 2-torsion. If $\alpha \ne 0$ and the
coefficient ring is $\mathbb{Q}$ then the homology of $P_2$ is
$2$-dimensional and isomorphic for any choice of $\alpha$. If $\alpha
\in \mathbb{Z}_+$ and the coefficient ring is $\mathbb{Z}$ then there is an
infinite family of $2$-torsion \emph{and} an infinite family of
$2\alpha$-torsion. In particular, the homotopy type of the projectors is
not constant with respect to the deformation parameter $\alpha$.

Taking the trace of our projector yields a complex with alternating
differential:
$$\begin{diagram}
\CPic{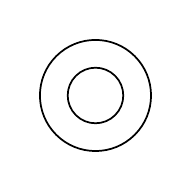} &
\rTo^{\hspace{.2in}\MPic{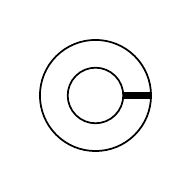}\hspace{.2in}} &
q \CPic{circle} &
\rTo^{\hspace{.2in}0\hspace{.2in}} &
q^{3} \CPic{circle} &
\rTo^{\hspace{.2in}2 \MPic{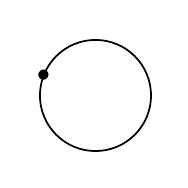}\hspace{.2in}} &
q^{5} \CPic{circle} & \rTo^{\hspace{.2in}0\hspace{.2in}} & \cdots
\end{diagram}$$

The homology of this complex is in general given by

$$
\Homology_n(\tr(P_2)) =
\left\{
\begin{array}{lr}
q^{-2}\mathbb{Z} \oplus q^0\mathbb{Z} & n = 0, \alpha = 0\textnormal{ or } \alpha \ne 0\\
0                                   & n = 1, \alpha = 0\textnormal{ or } \alpha \ne 0\\
q^{4k-2}\mathbb{Z}                    &n = 2k, \alpha  = 0\\
q^{4k+2}\mathbb{Z} \oplus q^{4k}\mathbb{Z}/2      & n = 2k+1, \alpha = 0\\
0                                   & n = 2k, \alpha \ne 0\\
q^{4k+2}\mathbb{Z}/(2\alpha) \oplus q^{4k}\mathbb{Z}/2      & n = 2k+1, \alpha \ne 0\\
\end{array}
\right.
$$

\medskip

\subsection{The Third Projector}

We give an inductive definition of the chain complex for the $n$th projector
in section \ref{proof section} below, with the second projector defined
above as the base of the induction. The third projector $P_3$ can therefore
be deduced from that inductive definition.  In this section we present a
minimal (in the sense that it is not homotopic to a chain complex containing
fewer Temperley-Lieb diagrams) chain complex for $P_3$.  The third projector
is the last which can be written down in a short and explicit diagrammatic
form. After the initial identity term the complex below becomes $4$
periodic.

$$\begin{diagram}
\MPic{n3-1}\, & \rTo^A &\, q^{1}\left(\! \MPic{n3-e1}\,\, \oplus \MPic{n3-e2}\!\right) & \rTo^B &\, q^{2}\left(\! \MPic{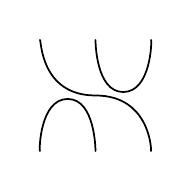}\,\, \oplus \MPic{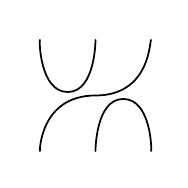}\!\right) & \rTo^C &\, q^{4}\left(\! \MPic{n3-e1e2}\,\, \oplus \MPic{n3-e2e1}\!\right) \\
  &  &           &   &         &   &  \dTo^D\\
\cdots & \lTo & \, q^{8}\left(\! \MPic{n3-e1e2}\,\, \oplus \MPic{n3-e2e1} \!\right) & \lTo^B & \, q^{7}\left(\! \MPic{n3-e1}\,\, \oplus \MPic{n3-e2}\!\right) & \lTo^E & \, q^{5}\left(\MPic{n3-e1}\,\, \oplus \MPic{n3-e2}\!\right) \\
\end{diagram}$$

Where $A = \left(\begin{array}{cc} \! -\MPic{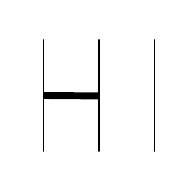} & \MPic{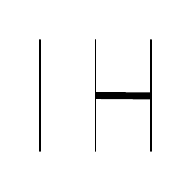} \end{array}\right)^{\rm tr}$ and

\begin{align*}
B &= \left(\begin{array}{cc}
\MPic{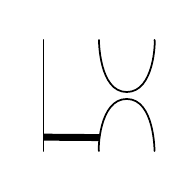} & - \MPic{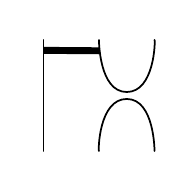}\\
-\MPic{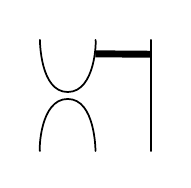} & \MPic{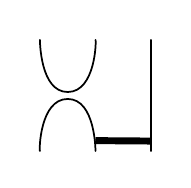}\\
\end{array}
\right)  &
C &= \left(\begin{array}{cc}
\MPic{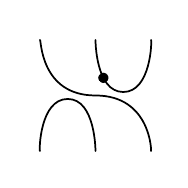} + \MPic{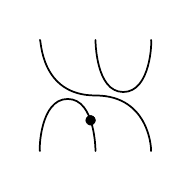} &  \MPic{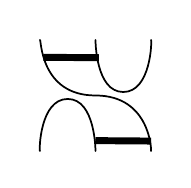}\\
\MPic{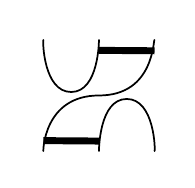} & \MPic{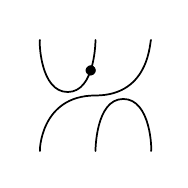} + \MPic{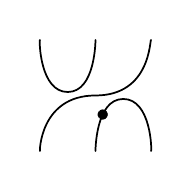}\\
\end{array}
\right) & \\
D &= \left(\begin{array}{cc}
\MPic{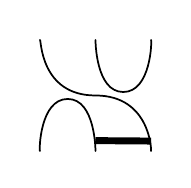} & - \MPic{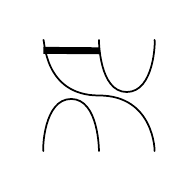}\\
-\MPic{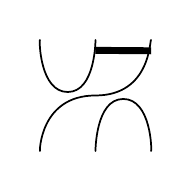} & \MPic{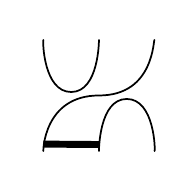}\\
\end{array}
\right) &
E &= \left(\begin{array}{cc}
\MPic{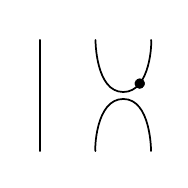} + \MPic{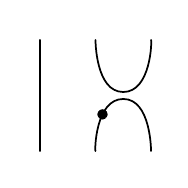} &  \MPic{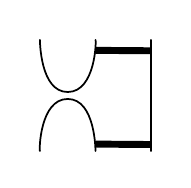}\\
\MPic{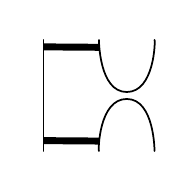} & \MPic{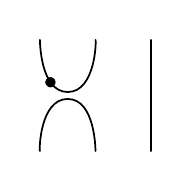} + \MPic{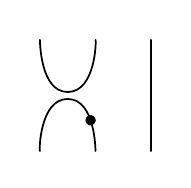}\\
\end{array}
\right) & \\
\end{align*}

\begin{theorem}
  The definition of $P_3$ given above is a chain complex that satisfies the
  axioms of the universal projector. In particular,
$$e_i\otimes P_3 \simeq 0 \simeq P_3\otimes e_i\hspace{1in}i = 1,2.$$
\end{theorem}

The proof is analogous to the proof of theorem \ref{secondproj thm}. The
main theorem also produces a universal projector $P_n$ for $n=3$.

\section{Reidemeister Moves and Graphical Calculus} \label{Reidemeister section}

In this section we define invariants of framed tangles obtained by applying
the $m$th projector to the strands of a cabling and showing that the
homotopy type of the result is invariant under Reidemeister moves 2 and 3.
These are categorifications of the invariants of higher representations of
$\quantsu$ corresponding to the colored Jones polynomial.

\begin{definition}
  Given $m\in\Mb{N}$ consider the category $\Kom_m(n)$ with objects

  $$\Obj(\Kom_m(n)) = \Obj(\Kom(n))$$

  To any object $D \in \Obj(\Kom(n))$ associate a chain complex $F(D)$ in
  the category $\Kom(mn)$ by replacing each strand in each diagram with $m$
  parallel strands composed with the $m$th projector.

  If $A$ and $B \in \Kom_m(n)$ are two objects then we define

  $$\Morph_{\Kom_m(n)}(A,B) = \Morph_{\Kom(mn)}(F(A),F(B))$$
\end{definition}

This can be illustrated by
$$\BPic{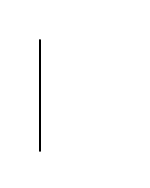} \hspace{-.4in}\mapsto \BPic{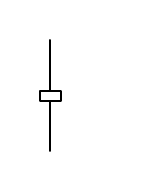}\hspace{1in}\BPic{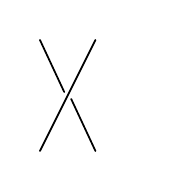}\hspace{-.2in} \mapsto \BPic{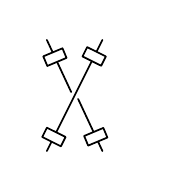}$$

In the remainder of this section we wish to prove that the Reidemeister
moves and some standard graphical relations are satisfied up to homotopy.

\begin{lemma}{(Projector Isotopy)} \label{projector isotopy}
  A free strand can be moved over or under a projector up to homotopy. In
  pictures,
$$\BPic{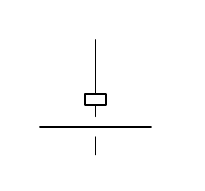} \simeq \BPic{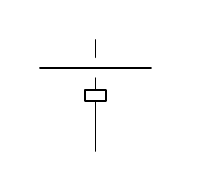}\hspace{1in}\BPic{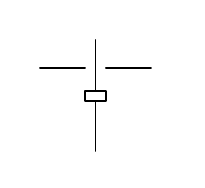} \simeq \BPic{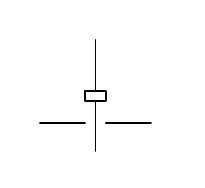}$$
\end{lemma}

\begin{proof} The proof is similar to the proof of proposition
  \ref{composition of projectors} and corollary \ref{homotopy uniqueness
    corollary} in section \ref{main theorem section}. Specifically, observe
  that both the chain complex for the diagram with the projector below the
  strand and the chain complex for the diagram with the projector above the
  strand are chain homotopy equivalent to the chain complex $C$ for the
  diagram with {\em two} projectors: one above the strand and one below the
  strand.  This is true because expanding either of the two projectors in
  $C$ gives the identity diagram in degree zero and every other term
  involves a turnback, which is contractible when combined with the second
  copy of the projector.
\end{proof}

Applications of this lemma allow us to show that the Reidemeister moves are
satisfied by the category $\Kom_m(n)$.

\begin{theorem}
  The category $\Kom_m(n)$ contains invariants of framed tangles.
\end{theorem}

\begin{proof}
Lemma \ref{projector isotopy} implies that up to homotopy projectors can be
slid under or over crossings. Corollary \ref{idempotency cor} implies that
up to homotopy any two projectors on the same strand can be replaced with
one projector.

For the second Reidemeister move,

$$\BPic{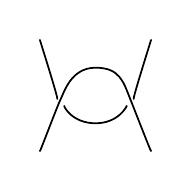}\quad = \BPic{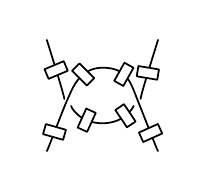}\quad \simeq \BPic{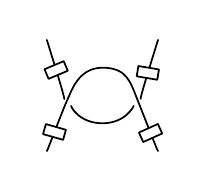} $$

The first equality is by definition. The homotopy equivalence follows from
the projector isotopy lemma and $P_n\otimes P_n \simeq P_n$.

$$\BPic{R2-3}\quad \simeq \BPic{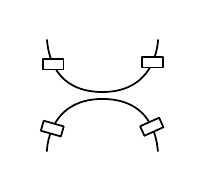}\quad \simeq \BPic{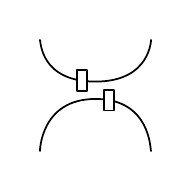}\quad = \BPic{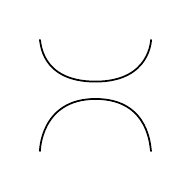}$$

The first homotopy equivalence follows from the second Reidemeister move in
$\Kom(mn)$, the second follows from another application of $P_n\otimes P_n
\simeq P_n$.

The argument for the third Reidemeister move features the same ideas.

$$\BPic{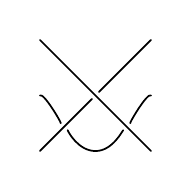} \quad = \BPic{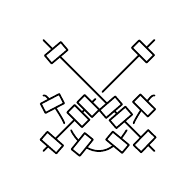}\quad \simeq\quad \BPic{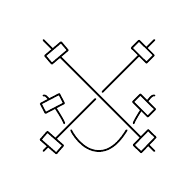} \quad\simeq \quad\BPic{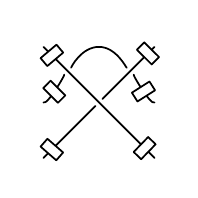} $$

$$\BPic{R3-4} \quad \simeq \BPic{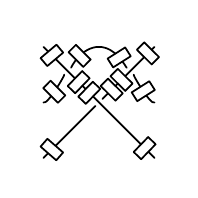}\quad = \BPic{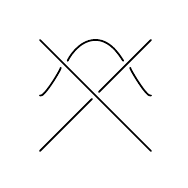} $$

Applying the definition to the standard Reidemeister 3 diagram we obtain a
diagram that looks like spaghetti which simplifies considerably up to
homotopy to a diagram in which the standard Reidemeister 3 homotopy in
$\Kom(mn)$ holds.
\end{proof}

Given a framed link $L$ the chain complex associated to $L$ in $\Kom_n(0)$
has the graded Euler characteristic of the positive $q$-power series
expansion of the $n$th colored Jones polynomial. This is a categorification
of the $n$th colored Jones polynomial. See section \ref{catdiscuss} for a
detailed discussion of what categorification means in this context.

\section{Spin Networks} \label{spin networks section}

In this section we describe how to associate to any spin network a chain
complex in a category defined using the universal projector of section
\ref{main theorem section}. Constructions involving four projectors are then
explored more thoroughly leading to a categorification of the 6j
symbols.

\subsection{Categories and Invariants}

Let $I=\Mb{N}$ be the set indexing the finite dimensional irreducible
representations of $\quantsu$. For any $n$-tuple $\mathbf{t} =
(i_1,\ldots,i_n) \in I^n$ define the invariants of the $n$-fold tensor
product by

$$\Inv(\mathbf{t}) = \Inv(V_{i_1} \otimes \cdots \otimes V_{i_n}) = \Morph_{\quantsu}(V_{i_1} \otimes \cdots \otimes V_{i_n},1)$$

This space is described by \emph{Temperley-Lieb diagrams} or unoriented
isotopy classes of 1-manifolds in a disk with boundary fixed in the boundary
of the disk and with boundary labeled by Jones-Wenzl projectors:
$p_{i_1}\sqcup p_{i_2} \sqcup \cdots \sqcup p_{i_n}$. See the illustration
below or also \cite{MR1280463, MR1403861}.

For any such $\mathbf{t} \in I^n$ the main theorem allows us to construct a
category $\Kom(\mathbf{t})$ with objects given by chain complexes obtained
from Temperly-Lieb diagrams with boundary labeled by universal projectors
and morphisms given by chain maps. When $\mathbf{t} = (a,b,c,d)$ there is an
associated picture,

$$\Obj(\Kom(\mathbf{t})) = \left\{ \BPic{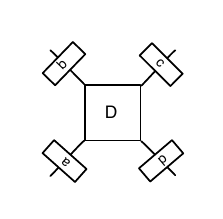}\quad : \textnormal{ D is a Temperly-Lieb diagram }\right\}$$

The axiomatic correspondence between the Jones-Wenzl projector and the
universal projectors in this paper implies the following theorem.

\begin{theorem}
The category $\Kom(\mathbf{t})$ categorifies the invariants $\Inv(\mathbf{t})$.
\end{theorem}

See section \ref{catdiscuss} for a detailed discussion of what
categorification means in this context.

\subsection{6j Symbols} \label{6j section}

There is a standard way to \emph{resolve} (compare \cite[Chapter 4.1]{MR1280463})
a trivalent vertex with edges labeled by $a,b,c \in I$ in a spin
network:

$$\BPic{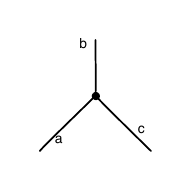} \quad = \BPic{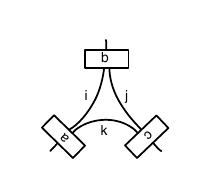}$$

where $i = (a+b-c)/2$, $j = (a+c - b)/2$ and $k = (b+c-a)/2$ are the number
of unoriented intervals placed between the projectors. We say that a diagram
is \emph{admissible} if all trivalent vertices can be resolved using the
assigned labels or equivalently $a+b+c$ is even and the triangle inequalites
hold for $a$, $b$ and $c$. Using this notation we can describe two bases for
$\Inv(a,b,c,d)$,

$$V = \left\{\CPic{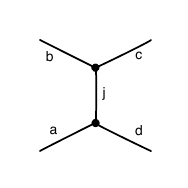} \quad : \begin{array}{c} j \in I\\ \textnormal{ admissible }\end{array}\right\} \textnormal{  and  } H = \left\{\CPic{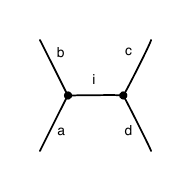} \quad : \begin{array}{c} i \in I\\ \textnormal{ admissible } \end{array}\right\}$$

The base change coefficients are called \emph{6j symbols},

$$\BPic{horzi} = \sum_j \left\{\begin{array}{ccc} a & b & i \\ c & d & j \end{array}\right\} \BPic{vertj}$$

which determine the change of basis map $S : H \to V$. $S$ is the matrix of
$6j$ symbols,
$$S_{ij} = \left\{\begin{array}{ccc} a & b & i \\ c & d & j \end{array}\right\} $$

{\bf Outline.} Our goal is to categorify $S$ as a functor $\Kom(a,b,c,d) \to
\Kom(a,b,c,d)$. Our construction is modeled on the linear-algebraic proof
(\cite{MR1280463} chapter 7.2) that the ``vertical'' and ``horizontal''
collections $V,H$ above are indeed bases for the space of Temperley-Lieb
diagrams $\Inv(a,b,c,d)$, pictured on the previous page.  The key point is
that the identity diagram appears only once in the chain complex for the
projector $P_n$, in degree zero (axiom (1) in definition \ref{universal
  projector def} of the universal projector). Therefore, the identity
diagram may be represented up to homotopy as the cone of the inclusion of
the positive degree part into the chain complex $P_n$. However the positive
degree part may in turn be inductively represented as an iterated cone on
lower order projectors. This is made precise in the proof of theorem
\ref{spin theorem} below. We begin by introducing a categorical analogue of
a linear basis.

Before proceeding we recall a number of definitions. The concept we wish to
capture is that of a category which is homotopy equivalent to some
subcategory sitting inside of it. In our case this amounts to a category of
complexes in which every chain complex is homotopy equivalent to a chain
complex of chain complexes contained within the subcategory of interest. All
of the categories involved in our discussion will be categories of chain
complexes of direct sums of objects in a small additive category.

A subcategory $\mathcal{C} \subset \mathcal{D}$ is \emph{full} if for all
pairs of objects $A,B \in \Obj(\mathcal{C})$,

$$\Morph_{\mathcal{C}}(A,B) = \Morph_{\mathcal{D}}(A,B)$$

A category $\mathcal{C}$ is \emph{differential graded} if for all objects
$A,B \in \Obj(\mathcal{C})$, $\Morph_{\mathcal{C}}(A,B)$ is a chain
complex. A functor of differential graded categories $F: \mathcal{C} \to
\mathcal{D}$ is a \emph{differential graded functor} if the maps

$$F_{A,B} : \Morph_{\mathcal{C}}(A,B) \to \Morph_{\mathcal{D}}(F(A),F(B))$$

are chain maps for all objects $A,B \in \mathcal{C}$.  Two differential
graded functors $F,G : \mathcal{C} \to \mathcal{D}$ are homotopic, $F \simeq
G$ if there is a natural transformation $\varphi : F \to G$ such that
$\varphi_A$ is a homotopy equivalence for all $A \in \Obj(\mathcal{C})$. Two
differential graded categories $\mathcal{C}$ and $\mathcal{D}$ are
\emph{homotopy equivalent} if there exist differential graded functors $F :
\mathcal{C} \to \mathcal{D}$ and $G : \mathcal{D} \to \mathcal{C}$ such that
$FG \simeq 1_{\mathcal{D}}$ and $GF \simeq 1_{\mathcal{C}}$.

\begin{definition}
  If $\mathcal{A}$ is an additive category and $\mathcal{C} =
  \Kom(\mathcal{A})$ is the category of chain complexes of finite direct
  sums of objects in $\mathcal{A}$ then a full subcategory $\mathcal{B}
  \subset \mathcal{C}$ \emph{spans} $\mathcal{C}$ if the inclusion
  $\mathcal{B} \hookrightarrow \mathcal{C}$ is a homotopy equivalence of
  categories.
\end{definition}

\begin{lemma}
For any $a,b,c,d \in \mathbb{N}$ the category $\Kom(a,b,c,d)$ is naturally
spanned by the full subcategory $\mathcal{N}$:

$$\Obj(\mathcal{N}) = \left\{ \CPic{abcddiag} \quad : \begin{array}{c} \normaltext{D is a $\TL$ diagram which contains no disjoint circles}\\ \normaltext{and no projector is capped by a turnback} \end{array} \right\}$$
\end{lemma}

\begin{proof}
Definition \ref{universal projector def} (3) (a) and (3) (b) implies that
projectors with turnbacks are contractible. Using the isomorphisms of
section \ref{categorified TL section} any complex associated to a diagram
with disjoint circles is isomorphic to a finite direct sum of chain
complexes associated to diagrams without circles. This defines a functor
$\pi : \Kom(a,b,c,d) \to \mathcal{N}$ which together with the obvious
inclusion functor $i : \mathcal{N} \hookrightarrow \Kom(a,b,c,d)$ satisfy
$\pi i = 1_{\mathcal{N}}$ and $i \pi \simeq 1_{\Kom(a,b,c,d)}$.
\end{proof}

There are two other categories we would like to consider: $\mathcal{H}$ and
$\mathcal{V}$. These are the full subcategories of $\Kom(a,b,c,d)$ with
objects given by horizontal and vertical diagrams respectively.
$$\Obj(\mathcal{H}) = \left\{\CPic{horzi} \quad : \begin{array}{c} i \in I\\ \textnormal{ admissible } \end{array}\right\} \textnormal{  and  } \Obj(\mathcal{V}) = \left\{\CPic{vertj} \quad : \begin{array}{c} j \in I\\ \textnormal{ admissible }\end{array}\right\}$$

We can now state our theorem,

\begin{theorem} \label{spin theorem}
  For any $a,b,c,d \in I$ the full subcategories $\mathcal{H}$ and $\mathcal{V}$
  defined above span the category $\Kom(a,b,c,d)$.
\end{theorem}

The proof consists of constructing a family of chain complexes $V_n$ and
$H_n$ each of which comes from the positive degree part of the chain complex
defining the $n$th universal projector $P_n$ constructed above. The gist of
the proof is captured by the two tables below in which
$(a,b,c,d)=(2,2,2,2)$.
$$\begin{array}{lr}
\underline{\underline{\mathcal{V}}} & \hspace{-.25in}\underline{\underline{\mathcal{N}}}\\
\Cone\left(V_4\left(\!\!\CPic{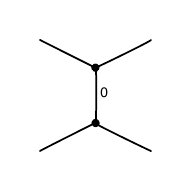}\;\,,\!\CPic{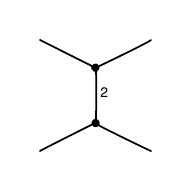}\;\,\right)\hookrightarrow \!\!\CPic{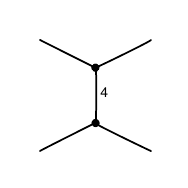}\;\,\right) & \simeq  \CPic{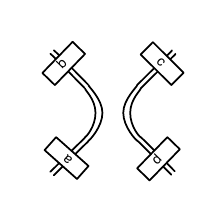}\quad  \\
\Cone\left(V_2\left(\!\!\CPic{vert0}\;\,\right)\hookrightarrow \!\!\CPic{vert2}\;\,\right) & \simeq  \CPic{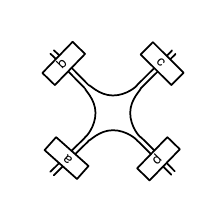}\quad \\
\Cone\left(V_0 \hookrightarrow \!\!\CPic{vert0}\;\,\right) & \simeq  \CPic{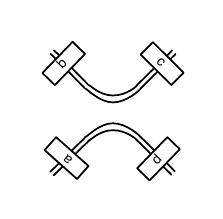}\quad  \\
\end{array}$$
\newline
$$\begin{array}{lr}
 \hspace{.25in}\underline{\underline{\mathcal{N}}}  & \underline{\underline{\mathcal{H}}} \qquad\\
\CPic{case0}\quad  \simeq &
\Cone\left(H_0 \hookrightarrow \!\!\CPic{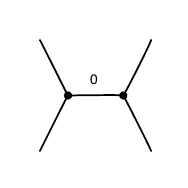}\;\,\right)\\
\CPic{case1}\quad  \simeq & \Cone\left(H_2\left(\!\!\CPic{horz0}\;\,\right)\hookrightarrow \!\!\CPic{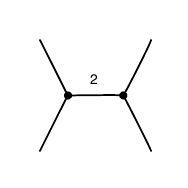}\;\,\right)\\
\CPic{case2}\quad  \simeq & \Cone\left(H_4\left(\!\!\CPic{horz0}\;\,,\!\CPic{horz2}\;\,\right)\hookrightarrow \!\!\CPic{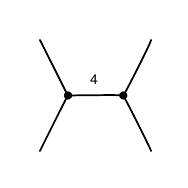}\;\,\right)
\end{array}$$

In the first table, the third chain complex $\MPic{case2}$ can be
constructed as a cone on the vertical spin network $\MPic{vert0}$. When
$(a,b,c,d) = (2,2,2,2)$ and $V_0 = 0$ the spin network itself is equal to
$\MPic{case2}$. Next, to define the chain complex $V_2$
expand the central projector of $\MPic{vert2}$ and then substitute
instances of $\MPic{case2}$ with $\Cone(V_0 \hookrightarrow
\MPic{vert0})$. The chain complex $V_2$ consists only of the resulting terms
with the $0$ labeled edge. The cone on the inclusion of $V_2$ into
$\MPic{vert2}$ is homotopic to $\MPic{case1}$. Finally, the first line
states that the object $\MPic{case0}$ is homotopic to the cone on a chain
complex $V_4$ consisting of only vertical networks labeled $0$ and $2$. We
now give a proof of theorem \ref{spin theorem}.

\begin{proof}
  We will explain the construction only for $\mathcal{V}$ since the argument
  is the same for $\mathcal{H}$.  For fixed $a,b,c,d \in I$ each diagram $D$
  defining an object in the spanning subcategory $\mathcal{N}$ defined above
  must be of the form
$$D_{j} = \BPic{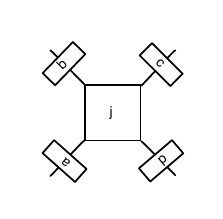}\quad := \BPic{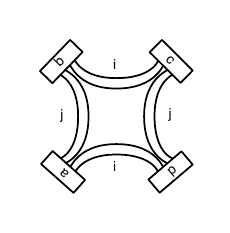}$$

where $j$ is the total number of vertical strands and $i$ is the total number of
horizontal strands. Notice $i$ depends on $j$ and if the diagram is
admissible the number $j$ assumes either odd or even integer values between
two numbers $l_0, l_N \in \mathbb{Z}_+$. If we use $D_n$ to denote the
diagram with $j = n$ in the illustration above then for a given $a,b,c,d \in
I$ we have diagrams $D_{l_0}, D_{l_0 + 2}, \ldots, D_{l_N -2}, D_{l_N}$. The proof is
by induction on the number of vertical strands.

For each $a,b,c,d \in I$ and admissible $j$ we define the chain complex
$V_j$,
$$V_j = V_j\left(\!\!\!\!\!\BPic{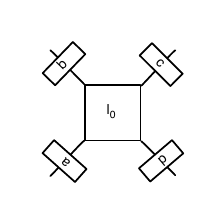}\quad,\,\cdots\,,\!\BPic{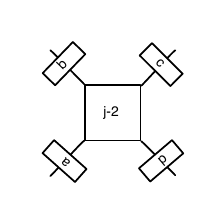}\quad\right)$$

to be the subcomplex shown below

$$\BPic{vertj} = \BPic{abcddiagj} \to V_j\left(\!\!\!\!\!\! \BPic{abcddiagl0}\quad,\,\cdots\,,\BPic{abcddiagjminus2}\quad\right)\quad $$

obtained by expanding the central projector $P_j$. As our notation suggests
$V_j$ is a chain complex containing only contractible terms and objects of
$\mathcal{N}$ defined by diagrams $D_{l_0},\ldots, D_{j-2}$.

The preceding diagram implies that the homotopy equivalence below is
tautological:
$$\BPic{abcddiagj} \simeq \Cone\left(V_j\left(\!\!\!\!\!\!\BPic{abcddiagl0}\quad,\,\cdots\,,\BPic{abcddiagjminus2}\quad\right) \overset{i}{\hookrightarrow} \!\!\!\!\BPic{vertj}\right)$$

where $i$ is the inclusion of the subcomplex $V_j$ into the complex
representing the labeled graph. This inclusion exists by construction of
$V_j$ above. The diagrams appearing in $D_0,\ldots, D_{j-2}$ have fewer
vertical strands and are homotopic to diagrams containing only vertical spin
networks by induction.

$$\BPic{abcddiagj} \simeq \Cone\left(V_j\left(\BPic{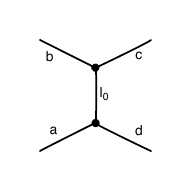},\,\cdots\,,\BPic{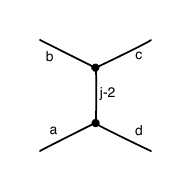}\right) \overset{i}{\hookrightarrow} \!\!\!\!\BPic{vertj}\right)$$

Again $i$ is the inclusion of the subcomplex $V_j$ into the complex
representing the labeled graph. The proof that substitution works is an
application of the change of basis isomorphism used in the Gaussian
elimination lemma in section \ref{homotopy lemmas} to the first vertical
identity map below.
\begin{align*}
\BPic{abcddiagj} &= \Cone\left(
\begin{diagram}
D_j     & \rTo^{d_1} & C_1  & \rTo^{d_2} & C_2 & \rTo  & \cdots\\
\uTo  &             & \uTo^1 &            & \uTo^1 &  & \\
 0     & \rTo       & C_1  & \rTo^{d_2} & C_2  & \rTo &  \cdots
\end{diagram}
\right) & \\
&\cong \Cone\left(
\begin{diagram}
D_j     & \rTo^{0} & C_1  & \rTo^{d_2} & C_2 & \rTo  & \cdots\\
\uTo  &             & \uTo^1 &            & \uTo^1 &  & \\
 0     & \rTo       & C_1  & \rTo^{0} & C_2  & \rTo &  \cdots
\end{diagram}
\right) & \\
\end{align*}

Note that in the base case, expanding the projector $p_{l_0}$ yields only
contractible terms in degree greater than zero. The $a,b,c,d,l_0$ labeled
network is homotopy equivalent to the cone on a nullhomotopic map of the
form above.

\end{proof}

\begin{remark}
  The naturally defined homotopy equivalence of categories $S:
  \mathcal{H} \to \mathcal{V}$ may be viewed as a categorical analogue of the matrix of 6j symbols
  $S : H \to V$ defined in section \ref{6j section}. See \ref{catdiscuss}
  for further discussion of what categorification means in this context.
\end{remark}

The quantum reader is invited to prove the homotopy
Biedenharn-Elliot identity.

\section{Proof of the Main Theorem} \label{proof section}

The two term recurrence relation satisfied by the Jones-Wenzl projectors in
section \ref{JW projectors} is quadratic in the sense that in order to define
$p_n$ the $n-1$st projector $p_{n-1}$ appears twice in the second term. One obtains the
linear recurrence of Frenkel and Khovanov \cite{MR1446615} by expanding
the bottom $p_{n-1}$ term completely and removing terms containing a turnback
$p_{n-1}e_i$ for any $0< i < n-1$. Keeping track of the
coefficients in this process gives the recurrence
$$\BPic{pn} \!= \!\BPic{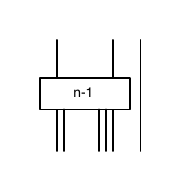} -\quad \frac{[n-1]}{[n]}\! \BPic{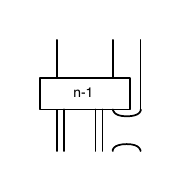} +\,\, \cdots\,\, \pm \quad \frac{[1]}{[n]}\! \BPic{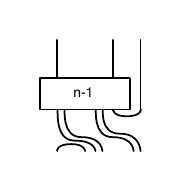}$$

Note that our sign conventions differ from \cite{MR1446615}.  This can be
shown to satisfy the axioms (1)-(3) in section \ref{JW projectors} and so is
equal to the Jones-Wenzl projector. In this section we prove the main
theorem of the paper by constructing a chain complex in the category
$\Kom(n)$, motivated by the Frenkel-Khovanov recursive formula above,
satisfying the axioms of the universal projector given in section \ref{main
  theorem section}.

\subsection{Triples and Quadruples} \label{triples and quadruples}

We will begin by examining some situations in which local cancelations can
be made in a chain complex containing a turnback: $C_* \otimes e_i$. There
are two important cases: either a sequence of three terms can be canceled
after delooping the middle term or a sequence of four terms can be canceled
after delooping the two middle terms. We will call the first case a
\emph{triple} and the second a \emph{quadruple}. Both cases are necessary to
prove $P_3 \otimes e_i \simeq 0$ and, as we will see, they suffice to prove
the general case.

\subsubsection{The Triple}\label{triple}

\begin{definition}
  If $D \in \Kom(n)$ is any chain complex and $e_i$ is a standard generator
  of $\TL_n$ then an \emph{$i$-triple} or \emph{triple} is a sequence of
  maps in $\Kom(n)$ of the form

$$ \begin{diagram}
 D & \rTo  & q D \otimes e_i  & \rTo  & q^2 D \otimes e_i \otimes e_{i \pm 1}
\end{diagram} $$

\end{definition}

 Where the maps are given by saddles, as illustrated below:
$$\begin{diagram}
\BPic{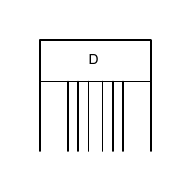} & \rTo^{\CPic{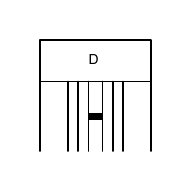}\quad} & q \BPic{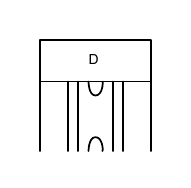} & \rTo^{\CPic{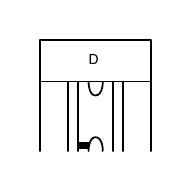}\quad} & q^2 \BPic{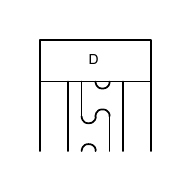}
\end{diagram}$$

Applying the functor $-\otimes e_i$ to the above yields the top sequence in
the following commutative diagram:

$$
\begin{diagram}
  D\otimes e_i & \rTo  & q D \otimes e_i \otimes e_i & \rTo  & q^2 D \otimes e_i \otimes e_{i \pm 1} \otimes e_i \\
\dTo^{\cong} & & \dTo^{\cong} & & \dTo^{\cong}\\
 D\otimes e_i & \rTo^{1 \oplus -}  & (D \otimes e_i) \oplus q^2 (D\otimes e_i) & \rTo^{- \oplus 1}  & q^2 D \otimes e_i
\end{diagram}$$

After applying $-\otimes e_i$ to an $i$-triple, the middle term can be
delooped (see remark \ref{deloop}) yielding the second isomorphism. The
last term satisfies the categorified planar isotopy relation 2 of section
\ref{TL section}, yielding the bottom sequence in the diagram above.  Note
that if this triple is part of a chain complex, then it can be canceled
using two applications of the Gaussian elimination (lemma \ref{gaussian
  elimination} in section \ref{homotopy lemmas}) or by a single application
of the simultaneous Gaussian Elimination (lemma 2).

\subsubsection{The Quadruple}\label{quadruple}

\begin{definition}
  If $D \in \Kom(n)$ is any chain complex and $e_i$ an elementary
  generator of $\TL_n$ then an \emph{$i$-quadruple} or \emph{quadruple} is
  a sequence of maps in $\Kom(n)$ of the form
$$
\begin{diagram}
D & \rTo  & q D \otimes e_i  & \rTo^A  & q^3 D \otimes e_i &\rTo & q^4 D \otimes e_i \otimes e_{i \pm 1}
\end{diagram}
$$

\end{definition}

All maps are given by saddles except for the $q$-degree 2 map $A$,
$$ A = \BPic{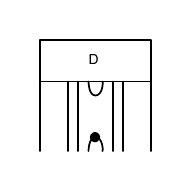}\quad -\BPic{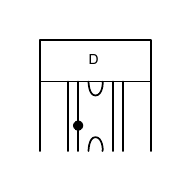}$$

$A$ is given by subtracting the addition of a handle on one strand from the
addition of a handle on an adjacent strand. Although this will not affect
the arguments below, we will fix the convention that the dot is placed
either to the left or the right of the saddle in the second term, depending
upon whether the $e_{i+1}$ or $e_{i-1}$ saddle is used at the end of the
quadruple.  The entire sequence of maps can be pictured as

$$\begin{diagram}
\CPic{triple1} & \rTo^{\MPic{triple1-m}} & q \CPic{triple2} & \rTo^{\MPic{triple2-m-quad}-\MPic{triple2-m-quad2} } & q^3 \CPic{triple2} & \rTo^{\MPic{triple2-m}} & q^4 \CPic{triple3}
\end{diagram}$$

Applying the functor $-\otimes e_i$ yields the top row of the following diagram:

{\Small
$$
\begin{diagram}
  D\otimes e_i  & \rTo  & q D \otimes e_i \otimes e_i  & \rTo^A  & q^3 D \otimes e_i \otimes e_i  &\rTo & q^4 D \otimes e_i \otimes e_{i \pm 1} \otimes e_i  \\
  \dTo^{\cong} & & \dTo^{\cong} & & \dTo^{\cong} & & \dTo^{\cong} \\
 D\otimes e_i & \rTo^{1 \oplus - } & (D \otimes e_i) \oplus q^2 (D \otimes e_i)  & \rTo^B  & q^2 (D \otimes e_i) \oplus q^4 (D \otimes e_i) & \rTo^{- \oplus 1}  & q^4 D \otimes e_i
\end{diagram}
$$
}

in which $B$ is of the form,

$$B = \left(\begin{array}{cc} - & 1 \\ - & - \end{array}\right)$$

Compare to delooping in Section \ref{second projector section} and
remark \ref{deloop}.  Again note that if this quadruple is part of a
chain complex, then it can be canceled using three applications of the
Gaussian elimination (or a single application of the simultaneous Gaussian
Elimination).

\subsection{The Frenkel-Khovanov Sequence}\label{fk sequence section}
The Frenkel-Khovanov formula from the beginning of this section suggests the
following recursive definition

$$\begin{diagram}[size=2em,height=.75in]
\SPic{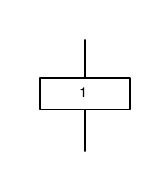} & = & \SPic{line} \!\!\!\!\!\! &\!\!\!\!\!\!\! \textnormal{ , } & \SPic{p2box} \textnormal{ is defined in \ref{second projector section} and for $n>2$ } & \SPic{pn} & = & \\
\end{diagram}$$

$$\begin{diagram}[size=2em,height=.75in]
q^0 \SPic{pnm1parp} & \rTo^{\CPic{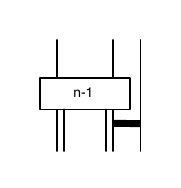}} & q^{1} \SPic{pnKF2} & \rTo^{\CPic{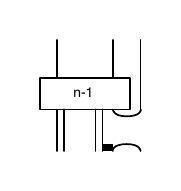}} & q^{2}\SPic{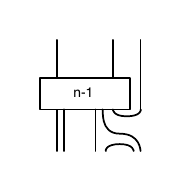} \\
                   &      &                     &      &    \dTo_{\CPic{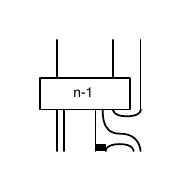}}            \\
q^{n-2}\SPic{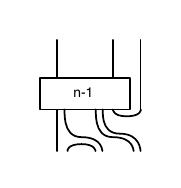} & \lTo^{\CPic{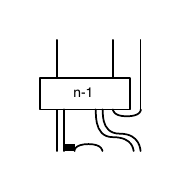}} &    \cdots       & \lTo^{\CPic{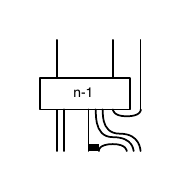}} & q^{3}\SPic{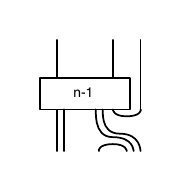}\\
\dTo^{\CPic{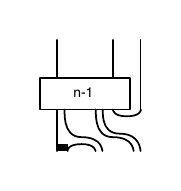}}  &      &                 &      &                   \\
q^{n-1}\SPic{pnKFn}   & \rTo^{\CPic{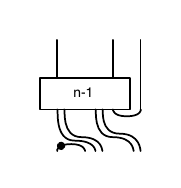}\,\,-\CPic{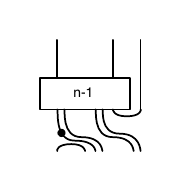}} & q^{n+1} \SPic{pnKFn} & \rTo^{\CPic{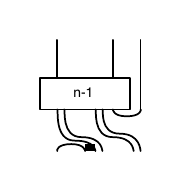}} & q^{n+2} \SPic{pnKFnm1} \\
                      &      &              &      & \dTo_{\CPic{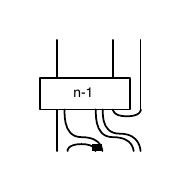}} \\
q^{2n-2}\SPic{pnKF3} & \lTo^{\CPic{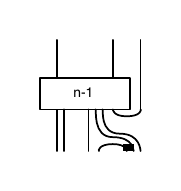}} & \cdots       & \lTo^{\CPic{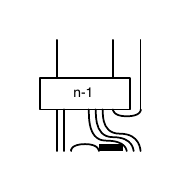}} & q^{n+3} \SPic{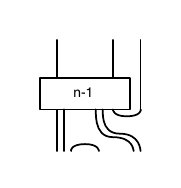} \\
\dTo^{\CPic{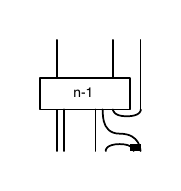}}    &      &              &      &               \\
q^{2n-1} \SPic{pnKF2} & \rTo^{\CPic{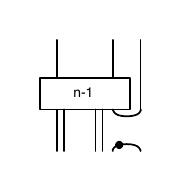}\,\,-\CPic{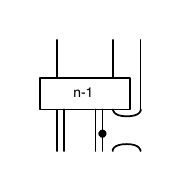}} &  q^{2n+1} \SPic{pnKF2} & \rTo & \cdots\\
\end{diagram}$$

Although the proposition below implies that this definition behaves
"correctly" with respect to the turnback axiom (3) of the universal
projector (see definition \ref{universal projector def} in section \ref{main
  theorem section}), the composition of two saddles is not equal to zero
(although it is homotopic to zero). The technical heart of this paper
consists of a detour taken purely for the purpose of arriving at an actual
chain complex. The final formulation obtained in section \ref{fattening fk
  section} will amount to a version of the above which has been carefully
thickened by contractible summands. See definition \ref{CFK definition} and
the picture following it. The reader is encouraged to check that the graded
Euler characteristic of the sequence illustrated above is a formal power
series corresponding to the Frenkel-Khovanov recursion formula for $p_n \in
\TL_n$.

In order to formalize the definition above consider the category \Mb{N}
determined by the graph

$$\begin{diagram}
0 & \rTo^{d^0} & 1 & \rTo^{d^1} & 2 & \rTo^{d^2}  & \cdots
\end{diagram}$$

The objects of \Mb{N} are non-negative integers and the morphisms are freely
generated by compositions of identity morphisms $1_i : i \to i$ and
morphisms $d^i : i \to i+1$.

\begin{definition}
  A \emph{sequence} $F$ in $\Kom(n)$ is a functor:
$$F : \Mb{N} \to \Kom(n)$$
\end{definition}

For each $n \geq 1$, we will define a sequence $\KF_n : \Mb{N} \to
\Kom(n)$. $\KF_n(k)$ will correspond to the bottom of each diagram in the
illustration on the previous page (the illustration itself equals
$(P_{n-1}\sqcup 1)\otimes \KF_n$). After the initial identity diagram,
$\KF_n$ is $2(n-1)$ periodic. The following is an algebraic definition of
the diagrams pictured above.

\begin{definition}\label{fk sequence}

  If $m \in \mathbb{Z}_+$ write $m = 2(n-1)q + r$ with $0 < r \leq 2(n-1)$
  then the $m$th diagram of the $n$th \emph{Frenkel-Khovanov sequence} is
  defined by

$$\KF_n(m) =
\left\{
\begin{array}{ll}
 1 & \textnormal{ if } m = 0 \\
q^m e_{n-1}\otimes\ldots\otimes e_{n-m}& \textnormal{ if } 1 \leq m < n \\
q^{2(m-n+1)} \KF_n(2n-m-1) & \textnormal{ if } n \leq m \leq 2(n-1) \\
q^{2n} \KF_n(r) & \textnormal{ otherwise }
\end{array}
\right.
$$

We use the multiplicativity of the formal $q$-grading: $q^i (q^j D) =
q^{i+j} D$ for any $D \in \Kom(n)$. The differential between any two objects
whose $q$-degree differs by one is given by a saddle map. In each period the
two $q$-degree $2$ differentials are defined to be those illustrated in the
diagram above. (The degree $2$ maps in the sequence are separated by $n-2$
saddle maps). In what follows $f_m$ will be used to denote the differential,
$f_m\co \FK_n(m)\longrightarrow \FK_n(m+1)$.

For a given length $l \geq 0 $ the $n$th \emph{truncated Frenkel-Khovanov
  sequence} is given by

$$\KF_{n,l}(m) = \left\{
\begin{array}{ll}
 \KF_n(m) & \textnormal{ if } m \leq l \\
0 & \textnormal{ otherwise }
\end{array}
\right.
$$
\end{definition}

The following proposition and its corollary are key ingredients in the proof
of contractibility under turnbacks contained in section \ref{fattening fk
  section}.

\begin{proposition} \label{turnbacks pass through prop}  Let $\KF_n : \Mb{N} \to \Kom(n)$ be the $n$th Frenkel-Khovanov sequence
  defined above. Then for any standard generator $e_i$, $i=1,\ldots,n-1$,

  $$\KF_n( - ) \otimes e_i : \Mb{N} \to \Kom(n)$$

  is a sequence such that for every $k\in\Mb{N}$ the diagram $\KF_n(k)
  \otimes e_i \in\Kom(n)$ either
\begin{enumerate}
\item Satisfies a commutativity condition: there exists a Temperley-Lieb
  element $D \in \Cob(n)$ such that

$$\KF_n(k) \otimes e_i \cong e_j \otimes D$$

where $j=i,i-1,\textnormal{ or } i-2$ or

\medskip

\item Is contained in an $i$-triple or $i$-quadruple sequence.

\end{enumerate}
\end{proposition}

The case (1) above when $j=i$ can be pictured by

$$\BPic{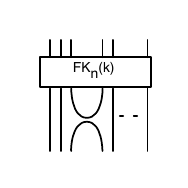}\quad \cong \BPic{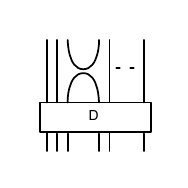}$$

Before giving the proof, we note that once $\FK_n$ is part of
the chain complex for the universal projector (defined further
below), both conclusions (1) and (2) above imply that all terms in
$(P_{n-1}\sqcup 1)\otimes \KF_n\otimes e_i$ may be contracted. In the case (1) this will
follow by the inductive contractibility of $P_{n-1}$ under turnbacks. The
contractibility in case (2) follows from the analysis of triples and
quadruples in section \ref{triples and quadruples}.

\begin{proof}
  The periodicity of the diagram $\KF_n$ ensures that inspecting the first
  $2n+1$ terms of $\KF_n\otimes e_i$ is sufficient.  Geometrically inclined
  readers are invited to prove the proposition by examining the illustration
  at the beginning of this section. Expanding $\KF_n$ allows us to write the first period of
  $\KF_n \otimes e_i$ as follows:

{\Small
  $$\begin{diagram}
1\otimes e_i & \rTo & e_{n-1}\otimes e_i & \rTo & (e_{n-1}\otimes e_{n-2})\otimes e_i &  \rTo &  \cdots & \rTo &  (e_{n-1}\otimes\cdots\otimes e_1)\otimes e_i & \\
 & & & & & & & & \dTo^2 \\
e_{n-1}\otimes e_i & \lTo^2 & e_{n-1}\otimes e_i &\lTo & \cdots & \lTo & (e_{n-1}\otimes\cdots\otimes e_2)\otimes e_i & \lTo & (e_{n-1}\otimes\cdots\otimes e_1)\otimes e_i
\end{diagram}$$
}

We have dropped the $q$-grading because it is implied by the requirement that
the first axiom (in definition \ref{universal projector def}) holds.  We write ``2'' above arrows in order to indicate
which maps are of $q$-degree $2$ and all other maps are given by saddles.

There are several cases to consider. The first two are boundary cases $i =
1$ and $i = n-1$ and the last is the generic case for $1 <i < n-1$.

\begin{enumerate}
\item If $i=1$ then consider $- \otimes e_1$. If $n-k > 2$ then because of
  the far commutativity relation,

$$(e_{n-1} \otimes e_{n-2} \otimes \cdots \otimes e_{n-k} ) \otimes e_1 \cong e_{1} \otimes (e_{n-1} \otimes e_{n-2} \otimes \cdots \otimes e_{n-k} )$$

$$\BPic{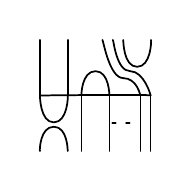} \quad\cong \BPic{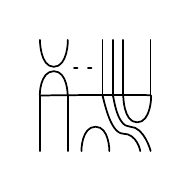}$$

The terms corresponding to $n-k = 1$ and $n-k = 2$ fit in the following four term sequence:

$$\begin{diagram}
\alpha \otimes e_1 & \rTo  &  \alpha\otimes e_1 \otimes e_1 & \rTo^2  & \alpha\otimes e_1 \otimes e_1  & \rTo & \alpha \otimes e_1,
\end{diagram}$$

\noindent
where $\alpha = e_{n-1} \otimes e_{n-2} \otimes \cdots \otimes e_{2} $. This sequence
forms a $1$-quadruple (see \ref{quadruple}).

\medskip

\item If $i = n-1$ then when $k > 2$ we have

$$(e_{n-1} \otimes e_{n-2} \otimes \cdots \otimes e_{n-k} ) \otimes e_{n-1} \cong e_{n-3} \otimes (e_{n-1} \otimes e_{n-2} \otimes \cdots \otimes e_{n-k} )$$

$$\BPic{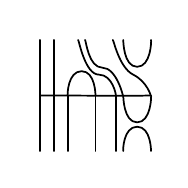} \quad\cong \BPic{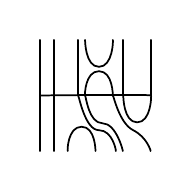}$$

When $k\leq 2$, the first three terms form an $(n-1)$-triple (see \ref{triple})

$$\begin{diagram}
 1\otimes e_{n-1} &  \rTo  & e_{n-1}\otimes e_{n-1} & \rTo & e_{n-1}\otimes e_{n-2} \otimes e_{n-1}.\end{diagram}$$

After the first period there is an $(n-1)$-quadruple surrounding every other
degree $2$ map:

$$\begin{diagram}
 e_{n-1} \otimes e_{n-2} \otimes e_{n-1} & \rTo & e_{n-1}\otimes e_{n-1}&  \rTo^{2} & e_{n-1}\otimes e_{n-1} & \rTo & e_{n-1} \otimes e_{n-2} \otimes e_{n-1}
\end{diagram}$$

\item If $i \ne 1$ and $i \ne n-1$,
each term has the form

$$ (e_{n-1} \otimes e_{n-2} \otimes \cdots \otimes e_{n-k} ) \otimes e_i $$

for some $k$ such that $2\leq k < n-1$. Depending on $k$ there are several cases to consider,
\begin{enumerate}
\item If $n - k > i + 1$ then $e_i$ commutes with $e_j$ for all $j$, $n-k
  \leq j \leq n-1$ because of the far commutativity relation. It follows
  that

$$ (e_{n-1} \otimes e_{n-2} \otimes \cdots \otimes e_{n-k} ) \otimes e_i \cong e_i \otimes (e_{n-1} \otimes e_{n-2} \otimes \cdots \otimes e_{n-k} )  $$

This can be pictured in the same way as the other application of the far
commutativity relation in (1).

\item If $n - k < i - 1$ then similarly,

$$(e_{n-1} \otimes e_{n-2} \otimes \cdots \otimes e_{n-k} ) \otimes e_i \cong e_{i-2} \otimes (e_{n-1} \otimes e_{n-2} \otimes \cdots \otimes e_{n-k} )$$

$$\BPic{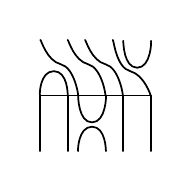}\quad \cong\BPic{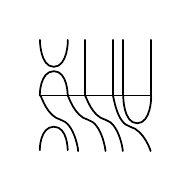}$$

\item The terms in which $n-k = i-1$, $n-k = i$ and $n-k = i+1$ form an
  $i$-triple. For instance,
$$\begin{diagram}
\cdots & \rTo &  \BPic{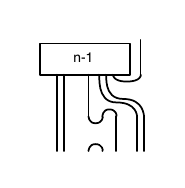} & \rTo & \BPic{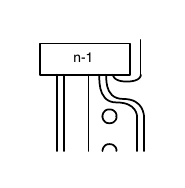} & \rTo & \BPic{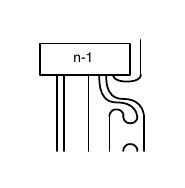} & \rTo & \cdots
\end{diagram}$$
\end{enumerate}
\end{enumerate}
\end{proof}

Let $f_m$ be the $m$th map in the Frenkel-Khovanov sequence, $$f_m\co
\FK_n(m)\longrightarrow \FK_n(m+1)$$ (see definition \ref{fk
  sequence}). Recall that each map $f_m$ has $q$-degree equal to either $1$
or $2$. All degree $1$ maps are given by saddle cobordisms and the degree 2
maps are shown in the diagram at the beginning of section \ref{fk sequence
  section}.

\begin{corollary}\label{truncated FK dies on turnback}
Suppose $P_{n-1} \in \Kom(n-1)$ is an $n-1$st universal projector and let
$l\in {\mathbb N}$ be such that $f_l$ is a degree $1$ map.  Then each term
in the truncated sequence

$$f_l \big( (P_{n-1} \sqcup 1) \otimes \FK_{n,l-1} \big)$$

either is the projector $P_{n-1}$ capped with a turnback or is
contained in a triple or a quadruple.
\end{corollary}

The proof of corollary \ref{truncated FK dies on turnback} follows from
proposition \ref{turnbacks pass through prop} since the map $f_l$ is assumed
to be a degree $1$ map, that is a saddle cobordism. Therefore the sequence
$f_l ( (P_{n-1} \sqcup 1) \otimes \FK_{n,l})$ equals $((P_{n-1} \sqcup 1)
\otimes \FK_{n,l})\otimes e_i$ for some $i$. \qed

The point of this corollary is that if $f_l \big( (P_{n-1} \sqcup 1) \otimes
\FK_{n,l} \big)$ is part of a chain complex, then it can be contracted. (In
the first case, the projector $P_{n-1}$ capped with a turnback is
contractible by axiom (1) of the universal projector $P_{n-1}$. In the
second case, each triple or quadruple is contractible according to the
analysis in sections \ref{triple}, \ref{quadruple}). This will play an
important role in the proof of the main theorem below.

\subsubsection{The homotopy projector} \label{homotopy projector section}
If an $n-1$st universal projector $P_{n-1}$ exists then corollary
\ref{truncated FK dies on turnback} shows that the Frenkel-Khovanov sequence
can be used to define a sequence which satisfies the axioms for an $n$th
universal projector (definition \ref{universal projector def}) up to
homotopy.

\begin{definition}\label{hty projector}
The $n$th \emph{homotopy projector} $HP_n : \Mb{N} \to \Kom(n)$ is the sequence defined by

\begin{align*}
  HP_1 &= 1\\
  HP_n &= \Tot\left( (P_{n-1} \sqcup 1) \otimes \KF_n \right)
\end{align*}

For a given length $l \geq 0$ the \emph{truncated homotopy projector} is
defined using the truncated Frenkel-Khovanov sequence:

$$ HP_{n,l} = (P_{n-1} \sqcup 1) \otimes \KF_{n,l} $$

\end{definition}

A picture of $HP_n$ is given at the beginning of this section. For each
$e_i$, $0< i < n$ by the above proposition $HP_n(k) \otimes e_i$ is either a
term containing $(HP_{n-1} \otimes e_j)\sqcup 1$ where $0<j<n-1$ or fits into
an $i$-triple or $i$-quadruple. If this were a chain complex then it would
be contractible by the lemmas of section \ref{homotopy lemmas}.

\subsection{Construction of the chain complex: fattening the FK sequence} \label{fattening fk section}

The remark at the end of section \ref{homotopy projector section} implies
that the sequence $HP_n$ (definition \ref{hty projector}) behaves like a
universal projector. However it is not a chain complex: the composition of
any two successive saddle maps is not zero (although it is not difficult to
see that all compositions are homotopic to zero).

In order to obtain a chain complex and so complete the proof of the main
theorem, we thicken the $\FK$ sequence by contractible pieces.
Specifically, we consider the truncated Frenkel-Khovanov
sequence $\FK_{n,l}$ of length $l$ and our construction is inductive in
$l$.

Let $P_{n-1} \in \Kom(n-1)$ be a chain complex representing the $n-1$st
universal projector. We will now define a chain complex $\CFK_{n,l}$
inductively in length $l$ using the maps $\{f_k\}_{0}^\infty$ of the
$\FK$-sequence. At each stage $\CFK_{n,l}$ is defined as either a two
term chain complex using $f_{l-1}$ (in case the following map $f_{l}$ has
$q$-degree $1$) or as a three term chain complex using $f_{l-1}$ and $f_{l}$
(in case $f_{l}$ has degree $2$).

\begin{definition}
\label{CFK definition}
Set $\CFK_{n,0} = P_{n-1} \sqcup 1$. For each $l>0$ the $q$-degree of $f_{l-1}$
is either $1$ or $2$. If $f_{l-1}$ has $q$-degree $1$ then set

$$\CFK_{n,l} =  \left\{
\begin{array}{ll}
\CFK_{n,l-1} \; \xrightarrow{f_{l-1}} \; q\, f_{l-1} \CFK_{n,l-1} &  \textnormal{ if } \deg_q(f_{l}) =1  \\
\CFK_{n,l-1} \; \xrightarrow{f_{l-1}} \; q\,  f_{l-1} \CFK_{n,l-1} \; \xrightarrow{f_{l}} \; q^3 \, f_{l} f_{l-1} \CFK_{n,l-1} &  \textnormal{ if } \deg_q(f_{l}) = 2\\
\end{array}
\right.
$$

Otherwise, the $q$-degree of $f_{l-1}$ is $2$ and we set

$$\CFK_{n,l} = \CFK_{n,l-1}.$$

In this second step we do not change the complex $\CFK_{n,l}$ after having
just used a degree 2 map in order to avoid a degree shift.

Here we follow the convention that $f_l(D\oplus D) = f_l(D) \oplus f_l(D)$,
$q^i(q^j D) = q^{i+j} D$ and $f_{l}(D) = D$ if $\deg_q(f_{l}) = 2$.
Note that the three term sequence in the last case is indeed a chain
complex, that is $f_{l} f_{l-1}=0$.

\end{definition}

The recursive step can be visualized as follows.

\begin{enumerate}
\item If $\deg_q(f_{l-1}) = 1$ and $\deg_q(f_{l}) =1$,

$$ \begin{diagram}
 \BPic{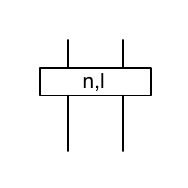} & =  & \BPic{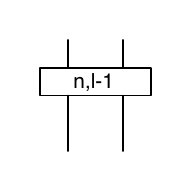}  & \rTo^{f_{l-1}}  & q \BPic{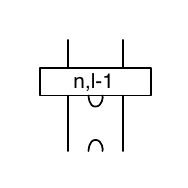}
\end{diagram} $$

\item If $\deg_q(f_{l}) =2$,

$$ \begin{diagram}
 \BPic{nlpic} & =  & \BPic{nlm1pic}  & \rTo^{f_{l-1}}  & q \BPic{nlm1sadpic} & \rTo^{f_{l}} & q^3 \BPic{nlm1sadpic}
\end{diagram} $$

\end{enumerate}

\newpage

Consider the chain complex $\CFK_{4,3}:$

\begin{diagram}[size=3em]
\CPic{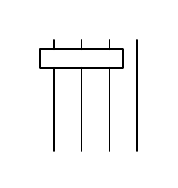}  &  \rDotsto^{f_0}            &         &       &  \CPic{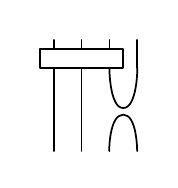} &              &   \\
        & \rdTo^{f_1}            &         &       & \dLine^{f_2} & \rdDotsto^{f_1}     &  \\
\dTo^{f_2}    &                 &  \CPic{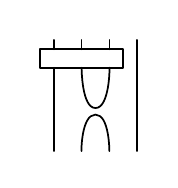}  & \rTo^{f_0}  &        &              &  \CPic{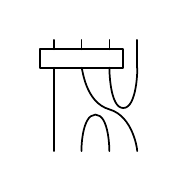}     \\
        &                  &         &       & \dTo   &              &        \\
\CPic{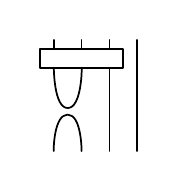}       &   \rLine    & \dTo^{f_2}       & \rTo^{f_0}  &  \CPic{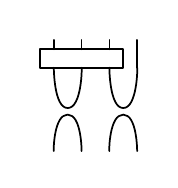}  &  & \dDotsto^{f_2}  \\
        &   \rdTo^{f_1}          &         &       &        & \rdTo^{f_1}      &       \\
        &                  & \CPic{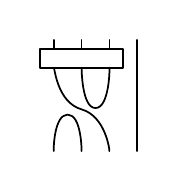}  &       &        &  \rTo^{f_0}        & \CPic{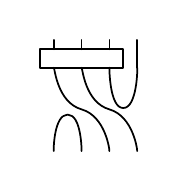}  \\
\end{diagram}

(To make the illustration more compact, only a part of the actual chain complex $\CFK_{4,3}$ is shown above. The actual 
$\CFK_{4,3}$ has another layer below the cube shown above.)
Recall that the truncated Frenkel-Khovanov sequence $\FK_{4,3}$ is given by

$$\begin{diagram}[center]
\CPic{c4id} & \rDotsto^{\!\!\!\!\! f_0} & \CPic{c4e3} & \rDotsto^{\!\!\!\!\!\! f_1} & \CPic{c4e3e2} & \rDotsto^{f_2} &\CPic{c4e3e2e1} &
\rDotsto^{f_3}  & \CPic{c4e3e2e1}
\end{diagram}
$$

Here we use dotted arrows to help the reader find the relevant
information. The first four terms of the sequence $\FK_{4,3}$ start in the
upper left hand corner of the cube, travel to the right then to the front
face of the cube and land in the lower right hand corner. These are precisely
the first four terms of the diagram pictured at the beginning of section
\ref{fk sequence section} together with four contractible terms.

The proofs contained in the remainder of this section are rooted in the
observation that $\CFK_{n,l}$ will always decompose as $\FK_{n,l}$ plus a
contractible subcomplex $K_l$ consisting of truncated Frenkel-Khovanov
sequences containing turnbacks.

\begin{lemma}\label{cfk structure} {(Structure of $\CFK_{n,l}$)}
For each $l\geq 0$ the chain complex $\CFK_{n,l}$ admits a decomposition,
\begin{equation} \label{cfk eq}
\CFK_{n,l} \cong (P_{n-1} \sqcup 1) \otimes (\FK_{n,l} \oplus K_l),
\end{equation}
where the second summand $(P_{n-1} \sqcup 1)\otimes K_l$ is contractible.

More specifically, the contractibility of $(P_{n-1} \sqcup 1)\otimes K_l$ is a consequence of
simultaneous Gaussian elimination of some of the terms in $K_l$,
so that each remaining term (in the notation of lemma \ref{sim gaussian elimination})
$(P_{n-1} \sqcup 1)\otimes C_i$ is contractible. Moreover, the off-diagonal component

$$\begin{diagram} (P_{n-1} \sqcup 1) \otimes K_l & \rTo^\beta & (P_{n-1} \sqcup 1) \otimes\FK_{n,l}\end{diagram}$$

of the $\CFK_{n,l}$ differential with respect to the
decomposition (\ref{cfk eq}) vanishes on the domain of the isomorphisms
underlying the simultaneous Gaussian elimination in $K_l$.
\end{lemma}

\begin{proof}
The proof is by induction on $l$ using the recurrence defining
$\CFK_{n,l}$ (definition \ref{CFK definition}). If $l = 0$ or
$\deg_q(f_{l-1}) = 2$ then there is nothing to prove. If $\deg_q(f_{l})
\ne 2$ then $\CFK_{n,l}$ is defined as a two term sequence,

$$\begin{diagram} \CFK_{n,l} &=& \CFK_{n,l-1} &\rTo^{f_{l-1}} & q f_{l-1} \CFK_{n,l-1}\end{diagram}$$

By induction we may assume that

$$\CFK_{n,l-1} \cong (P_{n-1} \sqcup 1) \otimes (\FK_{n,l-1} \oplus K_{l-1})$$

where $K_{l-1} \simeq 0$ satisfies the conclusion of the lemma. We claim that there is a decomposition

\begin{equation}\label{cfk iso}
\CFK_{n,l} \cong (P_{n-1} \sqcup 1) \otimes \Big( \FK_{n,l} \, \oplus \, q f_{l-1} \FK_{n,
l-2} \, \oplus \, K_{l-1} \, \oplus \, q f_{l-1} K_{l-1}\Big)
\end{equation}

This can be observed by writing the most important part of the recursion
defining $\CFK_{n,l}$ in a helpful way.

$${\Small \begin{diagram}
\Bigg[ & \FK_{n,l-1}(0) &  \rTo^{f_0} &  \cdots & \rTo^{f_{l-3}} &  \FK_{n,l-1}(l-2) & \rTo^{f_{l-2}} & \FK_{n,l-1}(l-1) \Bigg] & \subset & \CFK_{n,l-1} \\
       &  \dTo^{f_{l-1}}    &      & \cdots &      &  \dTo^{f_{l-1}} &   & \dTo^{f_{l-1}} & & \dTo^{f_{l-1}} \\
\Bigg[ & f_{l-1} \FK_{n,l-1}(0)  & \rTo^{f_0} & \cdots & \rTo^{f_{l-3}} & f_{l-1} \FK_{n,l-1}(l-2) \Bigg]&  \rTo^{\beta_l} &  f_{l-1} \FK_{n,l-1}(l-1) & \subset & f_{l-1} \CFK_{n,l-1}\\
& & & & & & &
\| & &  &  & \\
& & & & & & & \FK_{n,l}(l) & &  &  & \\
\end{diagram}}$$

In this diagram the lower order projector $(P_{n-1} \sqcup 1)$ is omitted to
simplify the notation. The terms in the top row are $\FK_{n,l-1}$, the
truncated Frenkel-Khovanov sequence of length $l-1$. By the inductive
assumption, this sequence is a summand in $\CFK_{n,l-1}$ and the remaining
part - the contractible summand $K_{l-1}$ - is not included in the diagram.
The terms on the left in the bottom row are of the form
$f_{l-1}\FK_{n,l-2}$. Observe that by definition the last term $f_{l-1}
\FK_{n,l-1}(l-1)$ is equal to the next term $\FK_{n,l}(l)$ in the
Frenkel-Khovanov sequence. The $\FK_{n,l}$ summand in (2) is seen in the
diagram above as $\FK_{n,l-1}$ in the top row followed by the vertical map
$f_{l-1}$ to $f_{l-1}\FK_{n,l-1}(l-1)$. The equation (\ref{cfk iso}) follows
immediately.

To prove that $(P_{n-1} \sqcup 1)\otimes K_l$ is contractible, where
$$K_l := q f_{l-1} \FK_{n, l-2} \, \oplus \, K_{l-1} \, \oplus \, q f_{l-1}
K_{l-1},$$ further analysis of the differential $\CFK_{n,l}$ is necessary.
Note that $(P_{n-1} \sqcup 1) \otimes f_{l-1}\FK_{n,l-2}$ is contractible by
corollary \ref{truncated FK dies on turnback}, more precisely all terms in
this sequence are either contained in triples or quadruples or are
projectors capped by turnbacks. By the inductive hypothesis on the
off-diagonal entry of the differential in the statement of the lemma, when
the terms participating in the Gaussian eliminations in $K_{l-1}$, $q f_{l-1}
K_{l-1}$ and $q f_{l-1} \FK_{n,l-2}$ are grouped together, (in the notation of
lemma \ref{sim gaussian elimination}) the isomorphisms underlying the
Gaussian eliminations in the summands, $a_{2i} : A_{2i} \to A_{2i + 1}$ and
$e_{2i+1} : B_{2i+1} \to B_{2i+2}$ remain isomorphisms because the matrices
are lower triangular.  After removing all triples and quadruples, the
remaining chain complex consisting of contractible terms may be contracted
by lemma \ref{big collapse} (big collapse). This concludes the proof that
$(P_{n-1} \sqcup 1)\otimes K_l$ is contractible.

To propagate the inductive hypothesis on the differential, note that the
only new off-diagonal component, of the form in the statement of the lemma,
introduced during the inductive step, is the map ${\beta}_l$ in the diagram
above.  The Gaussian eliminations take place in the sequence $q f_{l-1}
\FK_{n,l-2}$ in the bottom row and since ${\beta}_l$ is defined on the last
term of that sequence, clearly the component of the differential on the
domain of the isomorphisms in the Gaussian eliminations is trivial.

The proof in the second case (when $\deg_q(f_l) = 2$) is almost exactly
the same. Instead of one row of contractible terms there are two new rows of
contractible terms. By definition,

$$\begin{diagram} \CFK_{n,l} &=& \CFK_{n,l-1} &\rTo^{f_{l-1}} & q f_{l-1} \CFK_{n,l-1} & \rTo^{f_l} & q^3 f_l f_{l-1} \CFK_{n,l-1}
\end{diagram}$$

Again by induction we may assume that

$$\CFK_{n,l-1} \cong (P_{n-1} \sqcup 1) \otimes (\FK_{n,l-1} \oplus\, K_{l-1})$$

and $C_{l-1} \simeq 0$. In this case the claim is that there is a decomposition

$$\CFK_{n,l} \cong (P_{n-1} \sqcup 1) \otimes ( \FK_{n,l} \oplus\, K_l )$$

where

$$K_l := q f_{l-1} \FK_{n, l-2} \oplus K_{l-1} \oplus q f_{l-1} K_{l-1} \oplus q^3 f_l f_{l-1} \FK_{n,l-1} \oplus q^3 f_l f_{l-1} K_{l-1} $$

Again this can be observed by writing the most important part of the
recursion as follows:

$${\Small \begin{diagram}
\Bigg[ & \FK_{n,l-1}(0) &  \rTo^{f_0} &  \cdots & \rTo^{f_{l-3}} &  \FK_{n,l-1}(l-2) & \rTo^{f_{l-2}} & \FK_{n,l-1}(l-1) \Bigg] & \subset & \CFK_{n,l-1} \\
       &  \dTo^{f_{l-1}}    &      & \cdots &      &  \dTo^{f_{l-1}} &   & \dTo^{f_{l-1}} & & \dTo^{f_{l-1}} \\
\Bigg[ & f_{l-1} \FK_{n,l-1}(0)  & \rTo^{f_0} & \cdots & \rTo^{f_{l-3}} & f_{l-1} \FK_{n,l-1}(l-2) \Bigg]&  \rTo^{\beta_l} &  f_{l-1} \FK_{n,l-1}(l-1) & \subset & f_{l-1} \CFK_{n,l-1}\\
       &  \dTo^{f_l}    &      & \cdots &      &  \dTo^{f_l} &   & \dTo^{f_l} & & \dTo^{f_l} \\
\Bigg[ & f_l f_{l-1} \FK_{n,l-1}(0)  & \rTo^{f_0} & \cdots & \rTo^{f_{l-3}} & f_l f_{l-1} \FK_{n,l-1}(l-2) \Bigg]&  \rTo^{\beta_{l+1}} &  f_l f_{l-1} \FK_{n,l-1}(l-1) & \subset & f_l f_{l-1} \CFK_{n,l-1}\\
\end{diagram} }$$

As in the previous case, the summands in $K_l$ are contractible.
\end{proof}

{\bf Remark.}  The proof of the lemma above used the recursive definition of
the chain complex $\CFK_{n,l}$. Completely expanding the recursion gives the
following decomposition:

$$\CFK_{n,l} = \bigoplus_I q^{l(I) + \tau(I)} f_I \FK_{n,l-l(I)}$$

where $I$ are $k$-tuples indexing maps in the sequence $\FK_n$, $l(I)$ is
the cardinality of $I$, $f_I = f_{i_1} \circ f_{i_2}\circ \cdots \circ
f_{i_{k}}$ when $I = (i_1, i_2, \ldots, i_k)$ and $\tau(I)$ is the number of
degree $2$ maps in $f_I$.  Moreover, if $f_m$ is the differential of $\FK_n$
then the differential of in each summand, $f_I \FK_{n,l-l(I)}$, is
$f_I(f_m)$.  Each summand (except for $I=\emptyset$) is contractible by
corollary \ref{truncated FK dies on turnback}. In the notation of lemma
\ref{cfk structure}, $K_l=\bigoplus_{I\neq \emptyset} q^{l(I) + \tau(I)} f_I
\FK_{n,l-l(I)}$ and the summand $\FK_{n,l}$ corresponds to $I=\emptyset$.

The following statement is important for establishing the properties of a
universal projector:

\begin{lemma} {(Contractibility under turnbacks)} \label{turnbacks lemma}\nl
Let $n > 2$, $l \geq 0$ and $j\in\{1, \ldots, n-1\}$. Then all terms in the
chain complex
$$\CFK_{n,l} \otimes e_j$$ may be contracted, except possibly for the $l$th
term,

$$(P_{n-1} \sqcup 1) \otimes \FK_{n,l}(l)$$
\end{lemma}

\begin{proof}
By lemma \ref{cfk structure},
$$\CFK_{n,l} \otimes e_j  \cong \Big( (P_{n-1} \sqcup 1) \otimes (\FK_{n,l} \otimes e_j) \Big) \oplus \Big( (P_{n-1} \sqcup 1) \otimes (K_l \otimes e_j) \Big)$$
Application of $-\otimes e_j$ does not change the contractibility of the
the second summand.
By proposition \ref{turnbacks pass through prop} all of the terms (besides possibly the last, depending on $j$)
in the first summand are either projectors $P_{n-1}$ capped by
turnbacks or contained in triples and quadruples. The rest of the proof is identical to the
proof of the inductive step in lemma \ref{cfk structure}.
\end{proof}

As a consequence of lemma \ref{cfk structure} we have the following definition.

\begin{definition}{(The truncated projector $P_{n,l}$)}
Contracting $K_l \subset \CFK_{n,l}$ yields a homotopy equivalence

$$\begin{diagram} \CFK_{n,l} & \rTo  & P_{n,l} \end{diagram}$$

onto a \emph{chain complex} $P_{n,l}$ that consists of the first $l$ terms
of the Frenkel-Khovanov sequence pictured at the beginning of section \ref{fk
  sequence section}.
\end{definition}

Note that the chain complex $P_{n,l}$ may be thought of as a completed
version of the truncated homotopy chain complex $HP_{n,l}$ (definition
\ref{hty projector}). Given a homotopy chain complex there is a standard
obstruction theoretic approach to constructing a chain complex in which new
components corresponding to nullhomotopies and Massey products of
nullhomotopies are added to the differential (see \cite[2.10]{gelfandmanin}). The extra maps in $P_{n,l}$ are precisely those corresponding to
these homotopies and Massey products. Our axioms (definition \ref{universal
  projector def}) guarantee that any such choice of Massey products yields a
unique chain complex up to homotopy.

The universal projector $P_n$ will be defined as the limit of $P_{n,l}$ as
$l\rightarrow\infty$.  Its contractibility under turnbacks (to show that
$P_n$ satisfies the axioms of a universal projector in definition
\ref{universal projector def}) follows from lemma \ref{turnbacks lemma}.
The remaining property, ensuring that the limit exists, is the ``stability''
of the sequence $\{ P_{n,l}\}$, proved in the following proposition
\ref{stability}.

We now show that the chain complex $P_{n,l+1}$ is obtained from the chain
complex $P_{n,l}$ by adding the next term in the picture at the beginning of
section \ref{fk sequence section} and only adding maps to the differential
from the old terms to the new term. The maps between those terms in
$P_{n,l+1}$ which come from $P_{n,l}$ are exactly the same as the maps
between terms in $P_{n,l}$. We may conclude from this together with the
previous proposition that there is a chain complex $P_n = P_{n,\infty}$
which is a universal projector.

\begin{proposition}{(Stability of construction)} \label{stability}
The inclusion

$$P_{n,l} \hookrightarrow P_{n,l+1} $$

is an isomorphism onto its image. Moreover,

\begin{equation*}\label{pnl decom}
P_{n,l+1} \cong P_{n,l} \oplus \big ((P_{n-1} \sqcup 1) \otimes \FK_{n,l}(l) \big).
\end{equation*}

$d_{P_{n,l+1}}$ is lower triangular with respect to this decomposition and
$d_{P_{n,l+1}}|_{P_{n,l}} = d_{P_{n,l}}$.
\end{proposition}

\begin{proof}
This follows from an analogue of the first commutative diagram in the proof
of proposition \ref{cfk structure}. $P_{n,l+1}$ is obtained from $P_{n,l}$
by

$${\Small \begin{diagram}
\Bigg[ & P_{n,l}(0) &  \rTo^{d_0} &  \cdots & \rTo^{d_{l-3}} &  P_{n,l}(l-2) & \rTo^{d_{l-2}} & P_{n,l}(l-1) \Bigg] \\
       &  \dTo^{f_{l-1}}    &      & \cdots &      &  \dTo^{f_{l-1}} &   & \dTo^{f_{l-1}} \\
\Bigg[ & f_{l-1} P_{n,l}(0)  & \rTo^{d_0} & \cdots & \rTo^{d_{l-3}} & f_{l-1} P_{n,l}(l-2) \Bigg]&  \rTo^{\beta_l} &  f_{l-1} P_{n,l}(l-1) \\
\end{diagram}}$$

For the sake of clarity we have omitted from the diagram the parts of the
differential $d_{P_{n,l}}$ between non-consecutive terms.  The terms in the
lower lefthand corner are again contractible. Contracting them does not
change the maps $d_i$ along the top row.
\end{proof}

\subsection{A Doubling Construction}
In the proof of the main theorem we only concerned ourselves with what could
be called right contractibility or the statement that for $C_* \in\Kom(n)$
and $0< i < n$,

$$ C_* \otimes e_i \simeq 0 $$

If $C_*$ is right contractible then define $\bar{C}_* \in\Kom(n)$ to be the
chain complex in which each diagram and morphism is flipped upside down. Now
define a new chain complex $D_*$ by

$$ D_* = \bar{C}_* \otimes C_* $$

The contractibility of $D_*$ by turnbacks on both sides now follows from
that of $C_*$ on one side. The first two axioms of the universal projector
are satisfied by $D_*$ provided that they are satisfied by $C_*$.

\bibliographystyle{alpha}  
\bibliography{jwspinnet}  

\end{document}